\documentclass[11pt]{extarticle}

\usepackage[english]{babel}
\usepackage{graphicx}
\usepackage{framed}
\usepackage[normalem]{ulem}
\usepackage{amsmath}
\usepackage{amsthm}
\usepackage{dsfont}
\usepackage{amssymb}
\usepackage{amsfonts}
\usepackage{enumerate}
\usepackage[utf8]{inputenc}
\usepackage[top=1 in,bottom=1 in, left=1 in, right=1 in]{geometry}
\usepackage{mathrsfs}
\usepackage[nottoc, notlof, notlot]{tocbibind}
\usepackage{dsfont}
\usepackage{xcolor}
\usepackage{hyperref}
\usepackage{mathabx}
\usepackage{cases}
\usepackage{lineno}
\usepackage{stmaryrd}
\usepackage{setspace}

\newcommand{\mdiv}{{\rm div}}
\newcommand{\R}{\mathbb{R}}

\newcommand{\E}{\mathbb{E}}

\newtheorem{theorem}{Theorem}[section]

\newtheorem{lemma}{Lemma}[section]
\newtheorem{proposition}{Proposition}[section]

\theoremstyle{remark}
\newtheorem{remark}{Remark}

\theoremstyle{definition}
\newtheorem{definition}{Definition}[section]
\newtheorem{example}{Example}[section]

\DeclareMathOperator{\essup}{essup}

\title{Optimal control of the Fokker-Planck equation under state constraints in the Wasserstein space}
\author{Samuel Daudin \footnote{PSL Research University, Université Paris-Dauphine, CEREMADE, Place de Lattre de Tassigny, F- 75016 Paris, France, samuel.daudin@dauphine.eu}}

\begin{document}


\maketitle


\begin{abstract}
We analyze a problem of optimal control of the Fokker-Planck equation with state constraints in the Wasserstein space of probability measures. We give first-order necessary conditions for optimality in the form of a mean-field game system of partial differential equations associated with an exclusion condition. Under suitable geometric conditions on the constraint we prove that optimal controls are Lipschitz continuous.
\end{abstract}


\vspace{20pt}


\section*{Introduction}

This paper is devoted to the study of an optimal control problem of the Fokker-Planck equation under state constraints on the space of probability measures. The formulation of the problem is the following. We seek to minimize a cost 
\begin{equation}
J(\alpha,m):= \int_0^T \int_{\R^d} L \bigl(x, \alpha(t,x) \bigr) dm(t)(x)dt + \int_0^T \mathcal{F} \bigl(m(t) \bigr)dt + \mathcal{G}\bigl(m(T) \bigr)
\end{equation}
over pairs $(\alpha, m)$ with $m \in \mathcal{C}([0,T], \mathcal{P}_2(\R^d))$ and $\alpha \in L^2_{dt \otimes m(t)} \left( [0,T] \times \R^d, \R^d \right) $ (the  control) satisfying in the sense of distributions the Fokker-Planck equation:
\begin{equation}
\partial_t m + \mdiv (\alpha m ) - \Delta m = 0
\label{FPEintro2021}
\end{equation}
with the initial condition $m(0) = m_0 \in \mathcal{P}_2(\R^d)$. The flow of probability measures  $m$ is also constrained to satisfy the inequality
\begin{equation}
\Psi \bigl(m(t) \bigr) \leq 0, \hspace{20pt}  \mbox{           } \forall t \in [0,T]
\end{equation}
for some function $\Psi : \mathcal{P}_2(\R^d) \rightarrow \R$ satisfying additional conditions. Here $\mathcal{P}_2(\R^d)$ is the set of probability measures over $\R^d$ with finite second order moment. The functions $L : \R^d \times \R^d \rightarrow \R$ and $\mathcal{F} : \mathcal{P}_2(\R^d) \rightarrow \R$ are the running costs and $g: \mathcal{P}_2(\R^d) \rightarrow \R$ is the final cost.



Our first motivation comes from the theory of stochastic control. The corresponding problem is to minimize:
$$ \E \left[ \int_0^T L \bigl(X_t, \alpha_t \bigr)  dt+ \int_0^T \mathcal{F} \bigl(\mathcal{L}(X_t) \bigr)dt+ \mathcal{G}\bigl(\mathcal{L}(X_T) \bigr) \right]$$
over solutions of the stochastic differential equation $\displaystyle dX_t = \alpha_t dt + \sqrt{2}dB_t, $ where the controller starts from a random position $X_0$ with law $\mathcal{L}(X_0) = m_0$ and controls their drift $\alpha_t$ under the constraint $\Psi (\mathcal{L}(X_t)) \leq 0 $ for all $t \in [0,T]$. In this context, it is well-known that $\mathcal{L}(X_t)$, the law of $X_t$, solves Equation \eqref{FPEintro2021} in the sense of distributions and therefore the stochastic control problem reduces to a problem of optimal control of the Fokker-Planck equation (see \cite{Daudin2020} and the references therein). Stochastic optimal control problems with constraints on the probability distribution of the output have raised some interest in the past few years in connection with quantile hedging in \cite{Follmer1999}, stochastic target problems with \cite{Bouchard2009,Bouchard2010} and stochastic control problems with expectation constraints -see  \cite{Chow2020,Guo2022,Guo2019,Pfeiffer2020,Pfeiffer2020a} - to name a few. This problem was recently addressed in \cite{Frankowska2019} where the authors give first and second order necessary optimality conditions for stochastic control problems with state constraints in expectation form.

Our second motivation for studying constraints in law is that they arise, at least formally, as limit of symmetric, almost-sure constraints for stochastic control problems involving a large number of agents. The pre-limit problem would take the form
\begin{equation} 
\inf_{(\alpha_t^{i,N})_{1\leq i \leq N}} \E \left[ \int_0^T \frac{1}{N} \sum_{i=1}^N L(X_t^{i}, \alpha_t^{i}) dt + \int_0^T \mathcal{F}( \hat{m}_t^{N} ) dt +\mathcal{G}(\hat{m}_T^{N} )\right]
\label{NPIntro9Mars2023}
\end{equation}

\begin{equation*}
\left \{
\begin{array}{ll}
\displaystyle dX_t^{i} = b(X_t^{i}, \hat{m}_t^N,\alpha_t^{i})dt + \sqrt{2}dB_t^{i},  & \displaystyle \hat{m}_t^N = \frac{1}{N} \sum_{i=1}^N \delta_{X_t^{i}} \\
\displaystyle X^1_0,\dots,X^N_0 \mbox{   i.i.d.} \sim m_0
\end{array}
\right.
\end{equation*}
subject to 
$$ \Psi( \hat{m}_t^N) \leq 0 \hspace{30pt} \forall t\in [0,T], \hspace{20pt} \mbox{almost-surely}.$$
Almost-sure constraints in the case of non-degenerate diffusions are known to be difficult to handle. In particular, as shown in \cite{Lasry1989,Leonori2007}, the value function and the optimal controls blow-up near the boundary. We expect the analysis of Problem \eqref{NPIntro9Mars2023} to simplify by taking a limit as $N \rightarrow +\infty$.

Finally, we mention a motivation from the theory of large deviations for weakly interacting particles. Indeed, the asymptotic of rare event is  understood, in this setting, by the value of a mean-field control problem with constraints in law. More precisely, if one considers the particle system
\begin{equation*}
\left \{
\begin{array}{ll}
 dX_t^{i,N} = b(X_t^{i,N}, \hat{m}^{N}_t) dt +\sqrt{2}dB_t^{i,N} & \hat{m}_t^{N} = \frac{1}{N} \sum_{i=1}^N \delta_{X_t^{i,N}}, \\
X_0^{i,N} = x^{i,N}_0 \in \R^d, & \lim_{N \rightarrow +\infty} \hat{m}_0^N = m_0,
\end{array}
\right.
\end{equation*}
it is known from the seminal work of Dawson and Gärtner \cite{Dawson1987} that, under appropriate assumptions on $b$ and $\Psi$ the behavior as $N \rightarrow +\infty$ of the first exit time from $\left \{ \Psi \leq 0 \right \}$ for the empirical measure $\hat{m}_t^N$ when $\hat{m}_0^N \rightarrow m_0$ is given by
$$ \lim_{N \rightarrow + \infty} \frac{1}{N} \log \mathbb{P} \left[ \Psi \left( \frac{1}{N}\sum_{i=1}^N \delta_{X_t^{i}} \right) \leq 0, \forall t\in [0,T]  \right] = - \inf_{ (\alpha,m) } \int_0^T \int_{\R^d}  \frac{1}{2} |\alpha(t,x) |^2 dm(t)(x)dt  . $$
with the infimum taken over $(\alpha,m)$ solution to 
\begin{equation*}
\left \{
\begin{array}{ll}
\partial_t m + \mdiv(\left[ b(x,m(t)) + \alpha(t,x) \right] m) - \Delta m = 0 & \mbox{ in }(0,T) \times \R^d, \\
m(0)= m_0.
\end{array}
\right.
\label{FPELDPIntro}
\end{equation*}
under the constraint: $\displaystyle \Psi(m(t)) \leq 0, \forall t\in [0,T]$.

We refer to the forthcoming \cite{Daudin2023} for a precise discussion about these connections.

Given the type of constraints we are studying, here it is convenient to state our problem directly as an optimal control problem in the Wasserstein space. Such problems have been studied recently but mostly for control problems for the continuity equation (namely without diffusion term). Different approaches have been considered. In \cite{Jimenez2020,Marigonda2018} the authors use the dynamic programming approach and prove that the value function is the viscosity (in a sense adapted to the infinite dimensional setting) of an HJB equation. Whereas in \cite{Bonnet2019,Bonnet2021} the authors prove some adapted forms of the Pontryagin maximum principle. Notice that optimal control problems for the Fokker-Planck equation were previously considered in \cite{Carrillo2020,Fleig2017} but without constraint. Here we emphasize that the constraint is a smooth function defined on the Wasserstein space. In particular, our results do not cover the case of local constraints where the constraint acts on the density (when it exists) of $m$. This latter problem was addressed in \cite{Cardaliaguet2016,DiMarino2016b, Meszaros2015,Meszaros2018,Meszaros2016a}.

Here we follow the path initiated in \cite{Daudin2020} for a problem with terminal constraint and prove some optimality conditions in the form of a coupled system of partial differential equations associated with an exclusion condition. One of the equations is a Fokker-Planck equation satisfied by the solution of the problem. The other equation is a Hamilton-Jacobi-Bellman equation which is satisfied by an adjoint state, and from which we derive an optimal control. Besides these two equations, the exclusion condition reflects the effect of the constraint on the system. Our strategy is to proceed by penalization. We solve the penalized problem in a way that is closely related to Mean Field Game theory. Indeed, when the game has a potential structure - see for instance \cite{Briani2018,Cardaliaguet2015,Lasry2007,Orrieri2019} -  the system of partial differential equations which describes the value function of a typical infinitesimal player and the distribution of the players can be obtained as optimality conditions for an optimal control problem for the Fokker-Planck equation. With this optimality conditions at hand we proceed to show that solutions to the penalized problem -- when the penalization term is large enough-- stay inside the constraint at all times and are therefore solutions to the constrained problem. This second step is inspired by ideas in finite dimensional optimal control theory (see \cite{Frankowska2010}). In particular we follow a method used in \cite{Cannarsa2008,Cannarsa2018a}. The idea is to look at local maximum points of the function $t \mapsto \Psi(m(t))$ for some solution $m$ of the penalized problem and prove that they cannot satisfy $\Psi(m(t)) >0$ when the penalization is strong enough. To this end we compute the second order derivative of $t \mapsto \Psi(m(t))$ thanks to the optimality conditions previously proved. An interplay between the convexity of the Hamiltonian of the system, a tranversality assumption on the constraint and various estimates on the solutions of the optimality conditions of the penalized problem allows us to conclude. As a by-product of this method we can show that the solutions of the constrained problem enjoy the same regularity as the solutions of the penalized problem. In particular optimal controls are proved to be Lipschitz continuous. This result might seem surprising since the presence of state constraints generally leads to optimal controls which behave badly in time (see \cite{Frankowska2010} and the references therein). However it is reminiscent of classical results in finite dimensional optimal control theory in the presence of suitable regularity, growth and convexity assumptions as in see \cite{Galbraith2004,Hager1979}.

The rest of the paper is organized as follows. In Section \ref{sec: Preliminaries} we introduce the notations and state some useful preliminary results on the Fokker-Planck equation and the HJB equation on the one hand,  and on the differentiability of maps defined on the space of measures on the other hand. We also state a form of Itô's lemma for flows of probability measures. In Section \ref{sec: Main result and assumptions} we state the standing assumptions and our main results. In Section \ref{sec: The penalized problem} we obtain optimality conditions for the penalized problem. In Section \ref{sec: Proof of the main theorem} we prove our main theorem. In section \ref{sec: The general case} we extend our results to a more general setting.  Finally, we postpone to Section \ref{sec: Technical Results about the HJB equation} some technical results for the Hamilton-Jacobi equation satisfied by the adjoint state, that we use throughout the paper.

\paragraph{Notation}

For a map $u$ defined on $[0,T] \times \R^d$ we will frequently use the notation $u(t)$ to denote the function $x \mapsto u(t,x) $. Notice that $u(t)$ is therefore a function defined on $\R^d$. If a function $u$ defined on $[0,T] \times \R^d$ is sufficiently smooth, we denote by $\partial_t u$ the partial derivative with respect to $t$ and by $Du, \Delta u := \mdiv Du, D^2 u$ (if $u$ is a scalar function) or $Du, \overrightarrow{\Delta}u := \overrightarrow{\mdiv} Du$ if $u$ is vector-valued, the derivatives with respect to $x$.
The Wasserstein space of Borel probability measures over $\R^d$ with finite moment of order $r \geq 1$ is denoted by $\mathcal{P}_r(\R^d)$. It is endowed with the $r$-Wasserstein distance $d_r$. The space of $n$-times differentiable  bounded real functions over $\R^d$ with continuous and bounded derivatives is denoted by $\mathcal{C}_b^n(\R^d)$. Given $m \in \mathcal{C}([0,T], \mathcal{P}_2(\R^d))$ we denote by $L^2_{dt \otimes m(t)}([0,T] \times \R^d,\R^d)$ the space of $\R^d$-valued, $m(t) \otimes dt$-square-integrable functions over $[0,T] \times \R^d$.
The space of finite Radon measures over $[0,T]$ is denoted by $\mathcal{M}([0,T])$, the subset of non-negative measures by $\mathcal{M}^+([0,T])$ and the set of $\R^d$-valued Borel measures over $[0,T]\times \R^d$ with finite total variation by  $\mathcal{M}([0,T] \times \R^d , \R^d)$.
The space of symmetric matrices of size $d$ is denoted by $\mathbb{S}_d(\R)$.
We denote by $C_b^{1,2}([0,T] \times \R^d)$ the space of bounded functions with one bounded continuous derivative in time and two bounded continuous derivatives in space.
Finally we denote by $W^{1, \infty}([0,T] \times \R^d)$ the subspace of $L^{\infty}([0,T] \times \R^d)$ consisting of functions which have one bounded distributional derivative in space and one bounded distributional derivative in time.
For $n \geq 1$ we denote by $E_n$ the subspace of $\mathcal{C}^n(\R^d)$ consisting of functions $u$ such that 
$$ \|u\|_n := \sup_{x \in \R^d} \frac{|u(x)|}{1+|x|} + \sum_{k=1}^n \sup_{x \in \R^d} \left|D^k u(x) \right| <+ \infty. $$
Similarly we define $E_{n+\alpha}$ for $n \geq 1$ and $\alpha \in (0,1)$ to be the subset of $E_n$ consisting of functions $u$ satisfying
$$ \|u \|_{n +\alpha} := \|u \|_n + \sup_{x \neq y} \frac{|D^n u(x) - D^n u(y) |}{|x-y|^{\alpha}} < +\infty. $$ 
For $\alpha \in (0,1)$ we say that $u \in \mathcal{C}([0,T] \times \R^d)$ belongs to the parabolic Hölder space $\mathcal{C}^{(1+\alpha)/2,1+\alpha}([0,T] \times \R^d)$ if $u$ is differentiable in $x$ and 
\begin{align*}
 \| u \|_{\frac{1 +\alpha}{2}, 1+\alpha} &:= \sup_{(t,x) \in [0,T] \times \R^d} |u(t,x) | + \sup_{(t,x) \in [0,T] \times \R^d} |Du(t,x) | + \sup_{t \in [0,T]} \sup_{x \neq y} \frac{|Du(t,x) - Du(t,y)|}{|x-y|^{\alpha}} \\
 &+ \sup_{x \in \R^d} \sup_{t\neq s} \frac{|u(t,x) -u(s,x) |}{|t-s|^{(1+\alpha)/2}} + \sup_{x \in \R^d} \sup_{t\neq s} \frac{|Du(t,x) -Du(s,x) |}{|t-s|^{\alpha/2}}
 \end{align*}
is finite.
Finally we will use the heat kernel $P_t$ associated to $-\Delta$ defined, when it makes sense, by
$$ P_t f (x) := \int_{\R^d} \frac{1}{(4 \pi t)^{d/2}}e^{-\frac{|x-y|^2}{4t}}f(y)dy. $$

\section{Preliminaries}

\label{sec: Preliminaries}

We start by introducing the main protagonists of this paper. The first one is the Fokker-Planck equation.

\paragraph{The Fokker-Planck equation.}

Given $m \in \mathcal{C}([0,T], \mathcal{P}_2(\R^d))$ and $ \alpha \in L^2_{dt \otimes m(t)} \left( [0,T] \times \R^d, \R^d \right) $, we say that $(m, \alpha)$ satisfies the Fokker-Planck equation
\begin{equation}
\partial_t m + \mdiv(\alpha m ) -\Delta m =0
\label{FP}
\end{equation}
if for all $\varphi \in \mathcal{C}_c^{\infty}((0,T) \times \R^d) $ we have
\begin{equation}
\int_0^T \int_{\R^d} \left[ \partial_t \varphi(t,x) + D\varphi(t,x) . \alpha(t,x) + \Delta \varphi (t,x) \right] dm(t)(x)dt = 0.
\label{IPPDefinitionFokkerPlanck}
\end{equation}

Using an approximation argument similar to \cite{Trevisan2016} Remark 2.3,  we can extend the class of test functions to $\mathcal{C}_b^{1,2}([0,T] \times \R^d)$ and for all $\varphi \in \mathcal{C}_b^{1,2}([0,T] \times \R^d) $ and all $t_1,t_2 \in [0,T]$ it holds
\begin{align*}
\int_{\R^d} \varphi(t_2,x)dm(t_2)(x) &= \int_{\R^d} \varphi(t_1,x)dm(t_1)(x) \\
&+ \int_{t_1}^{t_2} \int_{\R^d} \left[ \partial_t \varphi(t,x) + D\varphi(t,x) . \alpha(t,x) + \Delta \varphi (t,x) \right] dm(t)(x)dt.
\end{align*}

Throughout the paper, we will repeatedly use the following properties of solutions to the Fokker-Planck equation. The proofs are given in the appendix.

\begin{proposition}
Assume that $m \in \mathcal{C}([0,T], \mathcal{P}_2(\R^d))$ and $ \alpha \in L^2_{dt \otimes m(t)} \left( [0,T] \times \R^d, \R^d \right) $ satisfy the Fokker-Planck equation \eqref{FP}, starting from the initial position $m_0 \in \mathcal{P}_2(\R^d)$ then,
$$\sup_{t \in [0,T]} \int_{\R^d} |x|^2dm(t)(x) + \sup_{t \neq s} \frac{d^2_2(m(t),m(s))}{|t-s|} \leq C$$
for some $\displaystyle C=C \Bigl(\int_{\R^d} |x|^2dm_0(x), \int_0^T \int_{\R^d} |\alpha(t,x) |^2dm(t)(x)dt \Bigr) >0.$
\label{EstimationsbaseFPE18Sept}
\end{proposition} 

We also have the following compactness result.

\begin{proposition}
Assume that, for all $k\geq 1$, $(m_k,\alpha_k)$ solves the Fokker-Planck equation \eqref{FP} starting from $m_0 \in \mathcal{P}_2(\R^d)$ and satisfies the uniform energy estimate
$$ \int_0^T \int_{\R^d} |\alpha_k(t,x) |^2dm_k(t)(x) dt \leq C, $$
for some $C>0$ independent of $k$. Then, for any $\delta \in (0,1)$, up to taking a sub-sequence,  $(m_k,\alpha_k m_k)$ converges in $\mathcal{C}^{\frac{1-\delta}{2}}([0,T], \mathcal{P}_{2-\delta}(\R^d)) \times \mathcal{M}([0,T] \times \R^d,\R^d)$ toward some $(m, \omega)$. The curve $m$ belongs to $\mathcal{C}^{1/2}([0,T], \mathcal{P}_2(\R^d))$, $\omega$ is absolutely continuous with respect to $m(t) \otimes dt$, it holds that
$$ \int_0^T \int_{\R^d} \left |\frac{d\omega }{dm(t)\otimes dt} (t,x) \right|^2 dm(t)(x)dt \leq \liminf_{k \rightarrow +\infty} \int_0^T \int_{\R^d} \left |\alpha_k(t,x) \right |^2 dm_k(t)(x)dt $$
and, finally, $(m, \frac{d \omega}{dt \otimes dm})$ solves the Fokker-Planck equation \eqref{FP} starting from $m_0$.
\label{CompactnessofcurvesFPE18Sept2022}
\end{proposition}

\paragraph{The HJB equation}

The second protagonist of this paper is the following Hamilton-Jacobi-Bellman equation. It involves the Hamiltonian $H : \R^d \times \R^d \rightarrow \R^d$ of the system. For the following definition to make sense and the next theorem to hold, $H$ is assumed to satisfy Assumption \eqref{AHallinone}, introduced in the next section.

\begin{definition}
Let $f \in L^1([0,T],E_n)$ and $g \in E_{n+\alpha}$ for some $n \geq 2$. We say that $u \in L^1([0,T], E_n)$ is a solution to 
\begin{equation}
\left \{
\begin{array}{ll}
-\partial_t u + H(x,Du) - \Delta u = f & \mbox{ in }[0,T] \times \R^d,\\
u(T,x) = g & \mbox{ in } \R^d,
\end{array}
\right.
\label{HJB12septembre2022}
\end{equation}
if, for $dt$-almost all $t\in [0,T]$ it holds, for all $x\in \R^d$
$$u(t,x) = P_{T-t}g(x) +\int_t^T P_{s-t}f(s)(x)ds - \int_t^T P_{s-t}\left[H(.,Du(s,.))\right] (x)ds. $$
\label{DefinitionHJB12Septembre2022}
\end{definition}
Let us point out that a solution $u \in \mathcal{C}([0,T],E_n)$ for $n\geq 3$ is differentiable in time whenever $f$ is continuous and, at these times, the HJB equation is satisfied in the usual sense.

We introduce this notion to handle solutions which are smooth in $x$ at each time but not necessarily regular in the time variable. 

The following theorem is proved in Section \ref{sec: Technical Results about the HJB equation}.

\begin{theorem}
Take $n\geq 2$. Assume that $f$ belongs to $L^1([0,T],E_n)$, g  belongs to $E_{n+\alpha}$ and $H$ satisfies Assumption \eqref{AHallinone} then,
\begin{itemize}
\item  The HJB equation \eqref{HJB12septembre2022} admits a unique solution $u$  in  $\mathcal{C}([0,T],E_n)$ in the sense of definition \ref{DefinitionHJB12Septembre2022}  and it satisfies the estimate
$$\sup_{t \in [0,T]}  \| u(t) \|_n \leq C(\int_0^T \|f(t) \|_ndt, \|g\|_n).$$
\item Assume that $(f_m,g_m)$ belongs to $L^1([0,T],E_n) \times E_{n+\alpha}$  for all $m\geq 1$ and that $f_m$ converges to  $f$ in $L^1([0,T],E_n)$ and $g_m$ converges to  $g$ in $E_{n+\alpha}$. Let $u_m$ be the solution to \eqref{HJB12septembre2022} with data $(f_m,g_m)$, then 
$u_m$ converges to  $u$ in $L^{\infty}([0,T] , E_n)$.
\end{itemize}
\label{TheoremHJB14Septembre2022}
\end{theorem}

\paragraph{Differentiability on the Wasserstein space and chain rule for flows of probability measures.}

We say that a map $U : \mathcal{P}_2(\R^d) \rightarrow \R^m$ is $\mathcal{C}^1$ if there exists a jointly continuous map $\displaystyle \frac{\delta U}{\delta m} : \mathcal{P}_2(\R^d) \times \R^d \rightarrow \R^m$ such that, for any bounded subset $\mathcal{K} \subset \mathcal{P}_2(\R^d)$, $\displaystyle x \rightarrow \frac{\delta U}{\delta m}(m,x)$ has at most quadratic growth in $x$ uniformly in $m \in \mathcal{K}$ and such that, for all $m, m' \in \mathcal{P}_2(\R^d)$,
$$U(m')-U(m)= \int_0^1 \int_{\R^d} \frac{\delta U}{\delta m} ((1-h)m+hm',x)d(m'-m)(x)dh. $$
The function $\displaystyle \frac{\delta U}{\delta m}$ is defined up to an additive constant and we adopt the normalization convention
$$\int_{\R^d} \frac{\delta U}{\delta m}(m,x)dm(x) = 0. $$
In the terminology of \cite{Carmona2018a} it means that $U$ admits a  linear functional derivative. 
When the map $\displaystyle x \mapsto \frac{\delta U}{\delta m}(m,x)$ is differentiable we define the intrinsic derivative of $U$
$$D_mU(m,x) := D_x \frac{\delta U}{\delta m}(m,x). $$

The following chain rule -formulated in terms of SDEs- is proved (under more general assumptions) in \cite{Carmona2018a} Theorem 5.99.
\begin{proposition}
Take $m \in \mathcal{C}([0,T], \mathcal{P}_2(\R^d))$ and $ \alpha \in L^2_{dt \otimes m(t)}\left( [0,T] \times \R^d, \R^d \right) $ such that $(m, \alpha)$ is a solution of the Fokker-Planck equation \eqref{FP} and suppose that $U : \mathcal{P}_2(\R^d) \times \R^d \rightarrow \R$ is $\mathcal{C}^1$ with $\displaystyle \frac{\delta U}{\delta m}$ satisfying
 $$ x \mapsto \frac{\delta U}{\delta m}(m,x) \in \mathcal{C}^2(\R^d), \hspace{25pt} \forall m \in \mathcal{P}_2(\R^d)$$
with $\displaystyle (m,x) \mapsto D_mU (m,x)$ and $\displaystyle (m,x) \mapsto D_xD_mU (m,x)$ being bounded on $\mathcal{P}_2(\R^d) \times \R^d$ and jointly continuous. Then, for all $ t \in [0,T]$, it holds that
\begin{align*}
U(m(t)) &= U(m(0)) + \int_0^t \int_{\R^d} D_m U(m(s),x).\alpha(s,x)dm(s)(x)ds \\
 &+\int_0^t \int_{\R^d} \mdiv_x D_m U(m(s),x)dm(s)(x)ds. 
\end{align*}
\label{ItoFlowMeasures}
\end{proposition}

Proposition 5.48 of \cite{Carmona2018a} ensures that $U$ satisfies the assumptions of Theorem 5.99.

\section{Main results and assumptions}

\label{sec: Main result and assumptions}

First, consider the unconstrained problem

\begin{equation}
\inf_{(\alpha,m)} J(\alpha,m),
\label{unconstrained}
\tag{uP}
\end{equation}
where 
$$J(\alpha,m) :=  \int_0^T \int_{\R^d}  L \bigl( x,\alpha(t,x) \bigr) dm(t)(x)dt + \int_0^T \mathcal{F} \bigl(m(t) \bigr) dt +  \mathcal{G} \bigl(m(T)\bigr) $$
is the total cost and the infimum runs over all $(\alpha,m)$ such that 
\begin{equation}
\left\{
\begin{array}{lr}
\displaystyle m \in \mathcal{C}([0,T], \mathcal{P}_2(\R^d)), \\
\displaystyle \alpha \in L^2_{dt \otimes m(t)} ([0,T] \times \R^d, \R^d), \\
\displaystyle \partial_t m + \mdiv (\alpha m) - \Delta m = 0 & \mbox{in   } (0,T)\times \R^d, \\
\displaystyle m(0) = m_0,
\end{array}
\right.
\label{Typicalcandidate}
\end{equation}
where the Fokker-Planck equation is understood in the sense of distributions. Here, the Lagrangian $L$ is defined by
$$L(x,q):= \sup_{p \in \R^d} \left \{ -p.q -H(x,p) \right \}$$
 and the data are the finite horizon $T>0$, the Hamiltonian $H : \R^d \times \R^d \rightarrow \R$, the mean-field costs $\mathcal{F}: \mathcal{P}_2(\R^d) \rightarrow \R$ and $\mathcal{G}: \mathcal{P}_2(\R^d) \rightarrow \R$ and the initial measure $m_0 \in \mathcal{P}_2(\R^d)$. The above data are supposed to satisfy the following conditions for some fixed integer $n \geq 3$.
 
For $\mathcal{U}=\mathcal{F}, \mathcal{G}$, the map $\mathcal{U}: \mathcal{P}_2(\R^d) \rightarrow \R^d$ satisfies 
\begin{equation}
\displaystyle \mbox{ $\mathcal{U}$ is a bounded from below, $\mathcal{C}^1$ map } \\
\mbox{ and $\displaystyle \frac{\delta \mathcal{U}}{\delta m} $ belongs to $\mathcal{C}(\mathcal{P}_2(\R^d),E_{n+\alpha}).$ }
\label{Newregularity2022}
\tag{Ureg}
\end{equation}
\begin{equation}
\left \{
\begin{array}{ll}
\displaystyle \mbox{$H$ belongs to $\mathcal{C}^n (\R^d \times \R^d)$.} \\
\displaystyle \mbox{$H$ and its derivatives are bounded on sets of the form $\R^d \times B(0,R)$ for all $R >0$.} \\
\displaystyle \mbox{For some $C_0>0$, for all $(x,p) \in \R^d \times \R^d$}, \\
\displaystyle \hspace{150pt} |D_xH(x,p)| \leq C_{0} (1 +|p|). \\
\displaystyle \mbox{For some $\mu >0$ and all $(x,p) \in \R^d \times \R^d$,} \\
\hspace{150pt}  \frac{1}{\mu} I_d \leq D^2_{pp}H(x,p) \leq \mu I_d.
\end{array}
\right.
\tag{AH}
\label{AHallinone}
\end{equation}
These assumptions imply in particular that $H$ has quadratic growth with respect to the $p$-variable. Taking convex conjugates, we see that $L$ satisfies a similar growth condition: for some $C>0$ and all $(x,q) \in \R^d \times \R^d$,
$$ \frac{1}{C}|q|^2 - C \leq L(x,q) \leq  \frac{C}{4}|q|^2 + C, $$
and the first term in the total cost $J$ looks very much like a kinetic energy.

A typical example of functions satisfying the condition \eqref{Newregularity2022} is the class of cylindrical functions of the form 
$$ \mathcal{F}(m) = F \left( \int_{\R^d} f_1(x)dm(x), \dots , \int_{\R^d} f_k(x)dm(x) \right), $$
where $F$ and the $f_i$, $1 \leq i \leq k$ are smooth with bounded derivatives.  Assumption \eqref{Newregularity2022} also implies that $(m,x) \rightarrow D_m \mathcal{U}(m,x)$ is uniformly bounded in $\mathcal{P}_2(\R^d) \times \R^d$ and therefore, a simple application of Kantorovitch-Rubinstein duality for $d_1$ proves that $\mathcal{U}$ is Lipschitz continuous with respect to this distance.  

Under the above assumptions on $\mathcal{F}$, $\mathcal{G}$ and $H$ it is well-known (see \cite{Briani2018, Daudin2020}), that solutions $(m,\alpha)$ of Problem \eqref{unconstrained} exist and satisfy $\alpha(t,x) = -D_pH(x,Du(t,x))$ with $(m,u)$ solution to the Mean-Field Game (MFG) system of partial differential equations
\begin{equation}
\left\{
\begin{array}{lr}
\displaystyle -\partial_t u(t,x) + H \bigl(x,Du(t,x) \bigr) - \Delta u(t,x) = \frac{\delta \mathcal{F}}{\delta m} \bigl(m(t),x\bigr) & \mbox{in   } (0,T)\times \R^d, \\
\displaystyle \partial_t m - \mdiv \bigl( D_pH(x, Du(t,x)) m \bigr) - \Delta m = 0  & \mbox{in   } (0,T)\times \R^d,  \\
\displaystyle u(T,x) =  \frac{\delta \mathcal{G}}{\delta m} \bigl(m(T),x \bigr) \hspace{10pt} \mbox{       in      } \R^d, \hspace{20pt}  \displaystyle m(0) = m_0, \\
\end{array}
\right.
\end{equation}
where the unknown $(u,m)$ belong to $\mathcal{C}^{1,2}((0,T) \times \R^d)$. 

The purpose of the present work is to investigate the effect of a state constraint 
$$\Psi \bigl(m(t) \bigr) \leq 0, \quad \forall t \in [0,T],$$ 
on the problem above.
Here $\Psi : \mathcal{P}_2(\R^d) \rightarrow \R$ satisfies the regularity assumption \eqref{Newregularity2022} and is convex for the linear structure of $\mathcal{P}_2(\R^d)$:
\begin{equation}
\Psi \mbox{ is convex.}
\label{PsiConvex}
\tag{APsiConv}
\end{equation}

We also need to assume that the problem is initialized at a point $m_0$ in the interior of the constraint that is
\begin{equation}
\Psi(m_0) <0.
\label{SartInside}
\tag{APsiInside}
\end{equation}

In addition to the previous assumptions we will ask for second-order differentiability with respect to the measure variable for $\Psi$.
\begin{equation}
\left \{
\begin{array}{ccc}
\displaystyle \mbox{For all $x \in \R^d$, $\displaystyle m \mapsto \frac{\delta \Psi}{\delta m}(m,x)$ is $\mathcal{C}^1$ with $\displaystyle (x,y) \mapsto \frac{\delta^2 \Psi}{\delta m^2}(m,x,y) := \frac{\delta^2 \Psi}{\delta m^2}(m,x)(y)$ } \\
\displaystyle \mbox{ in  $\mathcal{C}^2(\R^d \times \R^d)$ for all $m \in \mathcal{P}_2(\R^d)$ and $\displaystyle \frac{\delta^2 \Psi}{\delta m^2}(m,x,y)$ and its derivatives being} \\
\displaystyle \mbox{ jointly continuous and bounded in $\mathcal{P}_2(\R^d) \times \R^d \times \R^d$.}
\end{array}
\right.
\label{PsiC2}
\tag{APsiC2}
\end{equation}

Notice that Assumption \eqref{PsiC2} implies in particular (see for instance \cite{Carmona2018a} Remark 5.27) that the map $(m,x) \mapsto D_m\Psi(m,x)$ is  uniformly Lipschitz continuous over $\mathcal{P}_1(\R^d) \times \R^d$.

Finally we require the following geometric assumption on the constraint.
\begin{equation}
\displaystyle \int_{\R^d} \left|D_m \Psi(m,x) \right|^2dm(x) \neq 0 \mbox{ whenever } \Psi(m)= 0.
 \label{TransCondPsi}
 \tag{APsiTrans}
\end{equation}

The transversality assumption \eqref{TransCondPsi} is not necessary to get the optimality conditions however it is the key assumption to obtain the time regularity of optimal controls. Notice that \eqref{TransCondPsi} is satisfied as soon as $\Psi$ is displacement convex, there exists $m_0 \in \mathcal{P}_2(R^d)$ such that $\Psi(m_0)<0$ and $\Psi$ admits an intrinsic derivative.

An example of constraint $\Psi : \mathcal{P}_2(\R^d) \rightarrow \R$ satisfying Assumptions \eqref{Newregularity2022}, \ref{PsiConvex} and \ref{PsiC2} is $\displaystyle \Psi(m) := \int_{\R^d} \psi(x)dm(x) $ where $\psi$ is any function in $E_n$. If if holds as well that $|D\psi(x) | \neq 0$ whenever $\psi(x) \geq 0$ then $\Psi$ satisfies Assumption \eqref{TransCondPsi}. Indeed if $\displaystyle \int_{\R^d} |D\psi(x) |^2dm(x) = 0$ then $m$ must be concentrated on the set of points in $\R^d$ where $\psi(x) <0$ and therefore it cannot be that $\displaystyle \int_{\R^d} \psi(x) dm(x) = 0$.

\begin{example}
A typical example which satisfies Assumptions \eqref{Newregularity2022}, \eqref{PsiConvex}, \eqref{PsiC2} and \eqref{TransCondPsi} that we have in mind is $\displaystyle \Psi(m) = \int_{\R^d} \left(\sqrt{ |x-x_0|^2 +\delta^2 } - \delta \right) dm(x) - \kappa $ with $x_0 \in \R^d$, $\delta >0$ and $\kappa >0$. 




\label{TheOneTheOnlyOneExemple}
\end{example}

We can finally state the main problem of interest in this paper: 
\begin{equation}
 \inf_{(\alpha,m)} \int_0^T \int_{\R^d} L \bigl(x, \alpha(t,x)\bigr)  dm(t)(x)dt +\int_0^T \mathcal{F}(m(t))dt + \mathcal{G}(m(T))
 \tag{P}
\label{Problem}
\end{equation}
where the infimum runs over the pairs $(m,\alpha)$ satisfying \eqref{Typicalcandidate} and the state constraint 
\begin{equation*}
\Psi(m(t)) \leq 0, \hspace{20pt} \forall t \in [0,T].
\end{equation*} 

Over the course of the paper we will introduce several auxiliary problems. The main one is the following. For $\epsilon, \delta >0$ the penalized problem \eqref{PenalizedProblem} is
\begin{equation} 
\inf_{(m,\alpha)} J_{\epsilon, \delta} (\alpha, m)
\tag{$P_{\epsilon,\delta}$}
\label{PenalizedProblem}
\end{equation}
where the infimum runs over all $(m,\alpha)$ satisfying \eqref{Typicalcandidate} (but not necessarily the state constraint) and $J_{\epsilon, \delta}$ is defined by
\begin{align*}
J_{\epsilon, \delta}(\alpha,m) &:= \int_0^T \int_{\R^d}  L \bigl( x,\alpha(t,x)\bigr) dm(t)(x)dt + \int_0^T \mathcal{F}\bigl(m(t)\bigr) dt + \frac{1}{\epsilon} \int_0^T \Psi^+\bigl(m(t)\bigr) dt \\ & +  \mathcal{G}\bigl(m(T)\bigr) + \frac{1}{\delta}\Psi^+\bigl(m(T)\bigr) \\
&= J(\alpha,m) + \frac{1}{\epsilon} \int_0^T \Psi^+\bigl(m(t)\bigr) dt + \frac{1}{\delta}\Psi^+\bigl(m(T)\bigr).
\end{align*}
Here and in the following, $\Psi^+(m) = \Psi(m) \vee 0 = \max (\Psi(m),0)$. Notice that Problem \eqref{PenalizedProblem} is very similar to Problem \eqref{unconstrained} although we have to deal with the non-differentiability at $0$ of the map $r \mapsto \max(r,0)$.

We now state our main results. The first one is not expected without Assumption \eqref{TransCondPsi}. Roughly speaking, it asserts that optimal solutions to the penalized problems \eqref{PenalizedProblem}  stay inside the constraint when the penalization is strong enough.
\begin{theorem} Take $n\geq 3$. Assume that \eqref{AHallinone} holds for $H$,  \eqref{Newregularity2022} holds for $\mathcal{F}$ and $\mathcal{G}$. Assume further that $\Psi$ satisfies Assumptions \eqref{Newregularity2022}, \eqref{PsiConvex}, \eqref{SartInside}, \eqref{PsiC2} and \eqref{TransCondPsi}. Then there exist $\epsilon_0, \delta_0 >0$ depending on $m_0$ only through the value $\Psi(m_0)$ such that, for all $(\epsilon, \delta)$ in $(0,\epsilon_0) \times (0,\delta_0)$ Problems \eqref{PenalizedProblem} and \eqref{Problem} have the same solutions.
\label{MainTheoremStayInside}
\end{theorem}

As a consequence we find the following optimality conditions for the optimal control problem with constraint. 
\begin{theorem}
Under the same assumptions as Theorem \ref{MainTheoremStayInside},  Problem \eqref{Problem} admits at least one solution and, for any solution $(\alpha, m)$ there exist $u \in \mathcal{C}([0,T] , E_n ) $, $\nu \in L^{\infty} ([0,T])$ and $\eta \in \R^+$ such that 
\begin{equation} 
\alpha = -D_pH(x, Du)
\label{OptimalAlpha}
\end{equation}
and 
\begin{equation}
\left\{
\begin{array}{lr}
\displaystyle -\partial_t u(t,x) + H\bigl(x,Du(t,x)\bigr) - \Delta u(t,x) \\
\displaystyle \hspace{90pt} = \nu(t) \frac{\delta \Psi}{\delta m}\bigl(m(t),x \bigr) + \frac{\delta \mathcal{F}}{\delta m}\bigl(m(t),x \bigr) & \mbox{in   } (0,T)\times \R^d, \\
\displaystyle \partial_t m - \mdiv \bigl( D_pH(x, Du(t,x)) m \bigr) - \Delta m = 0  & \mbox{in   } (0,T)\times \R^d,  \\
\displaystyle u(T,x) = \eta \frac{\delta \Psi}{\delta m}\bigl(m(T),x \bigr) + \frac{\delta \mathcal{G}}{\delta m} \bigl(m(T),x \bigr)  & \mbox{       in      } \R^d, \\
\displaystyle m(0) = m_0,
\end{array}
\right.
\label{OptimalityConditionMainTheorem2021}
\end{equation}
where the Fokker-Planck equation is understood in the sense of distributions and $u$ solves the HJB equation in the sense of Definition \eqref{DefinitionHJB12Septembre2022} and the Lagrange multipliers $\nu$ and $\eta$ satisfy

\begin{minipage}{0.48 \textwidth}
\begin{equation}
\nu(t)  = \left \{ 
\begin{array}{ll} 
  0  & \mbox{if    } \Psi(m(t)) < 0 \\
  \nu(t) \in \R^+ & \mbox{if   } \Psi(m(t)) =0, 
\end{array}
\right.
\label{Exclusionnu}
\end{equation}
\end{minipage}
\begin{minipage}{0.48 \textwidth}
\begin{equation}
\eta  =\left \{ 
\begin{array}{ll} 
   0  & \mbox{if    } \Psi(m(T)) < 0 \\
 \eta \in \R^+ & \mbox{if   } \Psi(m(T)) =0. 
\end{array}
\right.
\label{Exclusioneta}
\end{equation}
\end{minipage}

In particular optimal controls are globally Lipschitz continuous in time and space.

If we also assume that $\mathcal{F}$ and $\mathcal{G}$ are convex in the measure variable, then the above conditions are sufficient conditions: if $(m, \alpha)$ satisfies $\Psi(m(t)) \leq 0$ for all $t \in [0,T]$ and if there exists $(u,\nu, \eta)$ such that \eqref{OptimalAlpha},  \eqref{OptimalityConditionMainTheorem2021}, \eqref{Exclusionnu} and \eqref{Exclusioneta} hold then $(\alpha,m)$ is a solution to \eqref{Problem}.
\label{MainTheorem2021}
\end{theorem}

The strength of the above result relies on the regularity of the Lagrange multiplier $\nu$ associated to the constraint that for all $t \in [0,T]$, $\Psi(m(t)) \leq 0$. Indeed we would a priori expect $\nu$ to be a finite Radon measure over $[0,T]$ but here we find that $\nu$ belongs to $L^{\infty}([0,T])$. As a consequence -- and as explained in Remark \ref{RkLipschitz} below-- optimal controls are Lipschitz continuous in time. 

We complete this section with a few comments.
\begin{remark}
Arguing as in \cite{Cannarsa2018a}, in the proof of Theorem 3.1, we can use the expression of $\frac{d^2}{dt^2}\Psi(m(t))$ given by Proposition \ref{PropDerivatives} to express $\nu(t)$ as a (non-local) feedback function of $Du(t)$, $D^2u(t)$ and $m(t)$. 
\end{remark}
\begin{remark}
Computing the cost of an optimal control we see that the value of the problem denoted by $\bar{\mathcal{U}}(m_0)$ is given by
$$\bar{\mathcal{U}}(m_0) = \int_{\R^d} u(0,x)dm_0(x) + \int_0^T \mathcal{F}(m(t)) dt + \mathcal{G}(m(T)) $$
for any solution $(m, -D_pH(x,Du))$ of \eqref{Problem}.
\end{remark}
\begin{remark}
Differentiating the HJB equation with respect to $x$ shows that $Du$ actually belongs to $W^{1,\infty}([0,T] \times \R^d,\R^d)$ and since $Du$ is also continuous and $D_pH$ Lipschitz continuous on $\R^d \times B(0,R)$ for all $R>0$, we get that $\alpha$ is Lipschitz continuous. In particular the Stochastic Differential Equation 
$$X_t = X_0 + \int_0^t \alpha(s,X_s)ds + \sqrt{2}B_t $$
where $X_0 \sim m_0$, admits a unique strong solution and we can proceed as in \cite{Daudin2020} to  find strong solutions to the stochastic analog of Problem \eqref{Problem} (as stated in the introduction).
\label{RkLipschitz}
\end{remark}
\begin{remark}
Ideally we would like to consider constraints of the form $\displaystyle \Psi(m)= \int_{\R^d} |x|^2dm(x) - \kappa$  (which does not satisfy the growth conditions of Assumptions \eqref{Newregularity2022} and \eqref{PsiC2})  for some $\kappa >0$. However this would significantly increase the technicality of the paper and we leave this case for future research. Among other difficulties we would have to solve the backward HJB equation in \eqref{OptimalityConditionMainTheorem2021} when the source term has a quadratic growth in the space variable.
\end{remark}
\begin{remark}
Our results could be naturally extended to multiple (possibly time dependent) equality or inequality constraints under suitable qualification conditions but we focus on this case of just one inequality constraint for the sake of clarity in an already technical paper. 
\end{remark}

\paragraph{Optimality conditions without Assumptions \eqref{PsiC2} and \eqref{TransCondPsi}. }

When Assumptions \eqref{PsiC2} and \eqref{TransCondPsi} are not satisfied we do not expect the conclusions of Theorem \ref{MainTheoremStayInside} to hold and therefore optimal controls might not be Lipschitz continuous. However, we can pass to the limit as $\epsilon,\delta$ go to $0$ in the Penalized problem \eqref{PenalizedProblem} and find the optimality conditions for the constrained problem.  This is the content of the next theorem.
\begin{theorem}
Assume that \eqref{AHallinone} holds for $H$,  \eqref{Newregularity2022} holds for $\mathcal{F}$ and $\mathcal{G}$. Assume further that $\Psi$ satisfies Assumptions \eqref{Newregularity2022}, \eqref{PsiConvex} and \eqref{SartInside}. Then the conclusions of Theorem \eqref{MainTheorem2021} hold true with $\nu \in \mathcal{M}^+([0,T])$, and $u \in L^{\infty}([0,T], E_n)$. The exclusion condition for $\nu$ now reads $\Psi(m(t)) = 0$, for $\nu$-almost all $t \in [0,T]$. Finally optimal controls belong to $BV_{loc}([0,T] \times \R^d,\R^d) \bigcap L^{\infty}([0,T], \mathcal{C}_b^{n-1}(\R^d,\R^d))$.
\label{MainTheoremWithoutTransCondPsi}
\end{theorem}

In this (slightly more) general case, we lose the time regularity of the optimal controls. This is due to the shocks that can occur when the optimal curve $t \rightarrow m(t)$ touches the constraint. Indeed, the set of times where the optimal control is not continuous, is contained into the support of the singular part of the Lagrange multiplier $\nu$. However, the space regularity of the backward component $u$ of the system and of the optimal control $-D_pH(x,Du)$ remains. 

The proof of Theorem \ref{MainTheoremWithoutTransCondPsi} is the aim of Section \ref{sec: The general case} where we discuss in particular the well-posedness of the HJB equation when the Lagrange multiplier $\nu$ belongs to $\mathcal{M}^+([0,T])$.

\section{The penalized problem}

\label{sec: The penalized problem}

In this section we analyze the penalized problem \eqref{PenalizedProblem}. The main result is the following.

\begin{theorem}
Problem \eqref{PenalizedProblem} admits at least one solution and, for any solution $(\alpha, m)$ of \eqref{PenalizedProblem} there exist $u \in  \mathcal{C}([0,T] , E_n ) $, $\lambda \in L^{\infty} ([0,T])$ and $\beta \in [0,1]$ such that $ \alpha = -D_pH(x, Du) $ and 
\begin{equation}
\left\{
\begin{array}{lr}
\displaystyle -\partial_t u(t,x) + H(x,Du(t,x)) - \Delta u(t,x) \\
\hspace{90pt} \displaystyle = \frac{\lambda(t)}{\epsilon} \displaystyle \frac{\delta \Psi}{\delta m}(m(t),x)  + \frac{\delta \mathcal{F}}{\delta m}(m(t),x) & \mbox{in   } (0,T)\times \R^d, \\
\displaystyle \partial_t m - \mdiv( D_pH(x, Du(t,x)) m ) - \Delta m = 0  & \mbox{in   } (0,T)\times \R^d,  \\
\displaystyle u(T,x) = \frac{\beta}{\delta} \frac{\delta \Psi}{\delta m}(m(T),x) +  \frac{\delta \mathcal{G}}{\delta m}(m(T),x) \hspace{5 pt} \mbox{       in      } \R^d, \hspace{20pt} \displaystyle m(0) = m_0. \\
\end{array}
\right.
\label{OCPen}
\end{equation}
Moreover, $\lambda$ and $\beta$ satisfy

\begin{minipage}{0.48 \textwidth}
\begin{equation}
\lambda(t)  \left \{ 
\begin{array}{ll} 
  =0  & \mbox{if    } \Psi(m(t)) < 0 \\
  \in [0,1] & \mbox{if    } \Psi(m(t)) = 0 \\
 = 1 & \mbox{if    } \Psi(m(t)) > 0, \\
\end{array}
\right.
\label{Exclusionlambda}
\end{equation}
\end{minipage}
\begin{minipage}{0.48 \textwidth}
\begin{equation}
\beta  \left \{ 
\begin{array}{ll} 
 =0  & \mbox{if    } \Psi(m(T)) < 0 \\
 \in [0 , 1] & \mbox{if    } \Psi(m(T)) = 0 \\
 =1 & \mbox{if    } \Psi(m(T)) > 0. \\
\end{array}
\right.
\label{Exclusionbeta}
\end{equation}
\end{minipage}

\label{TheoremPen}
\end{theorem}

The proof of Theorem \ref{TheoremPen} will be divided into three steps. First we are going to prove the existence of (relaxed) solutions to the problem. This is Lemma \ref{Existenceofweaksolutions2021}. In the second step, we will show that these relaxed solutions are actually solutions of a suitable linearized problem. This is Lemma \ref{Linearizationlemma}. Finally, we will conclude the proof of Theorem \ref{TheoremPen} by computing the optimality conditions for this linearized problem. The three steps above are very similar to what is done in \cite{Briani2018} Lemma 3.1 and in \cite{Daudin2020} Section 3. Here, however we have to deal with the lack of differentiability at $0$ of the function $r \mapsto \max(0,r)$. We also proceed differently at the end of the proof of Theorem \ref{TheoremPen}, where we argue by verification to avoid the unnecessary use of a min/max argument.

We start with the existence of relaxed solutions. A relaxed candidate is a pair $(m, \omega)$ such that 
\begin{equation}
\left\{
\begin{array}{lr}
\displaystyle m \in \mathcal{C}([0,T], \mathcal{P}_2(\R^d)), \\
\displaystyle \omega \in \mathcal{M}([0,T] \times \R^d,\R^d), \\
\displaystyle \partial_t m + \mdiv (\omega) - \Delta m = 0 & \mbox{in   } (0,T)\times \R^d, \\
\displaystyle m(0) = m_0,
\end{array}
\right.
\label{CandidateRelaxed}
\end{equation}
where the Fokker-Planck equation is once again understood in the sense of distributions. 

A relaxed solution is a minimizer over all the relaxed candidates of the following functional still denoted (with a slight abuse of notations) by $J_{\epsilon, \delta}$
\begin{align*}
J_{\epsilon, \delta}(m, \omega) &:= \int_0^T \int_{\R^d}  L \Bigl( x,\frac{d\omega}{dt \otimes dm(t)}(t,x) \Bigr) dm(t)(x)dt + \int_0^T \mathcal{F}(m(t)) dt + \frac{1}{\epsilon} \int_0^T \Psi^+(m(t)) dt \\ 
& +  \mathcal{G}(m(T)) + \frac{1}{\delta}\Psi^+(m(T)),
\end{align*}
where we set $J_{\epsilon, \delta}(m, \omega)  = +\infty$ if $\omega$ is not absolutely continuous with respect to $dt \otimes m(t)$.

\begin{lemma} 
Problem \eqref{PenalizedProblem} admits at least one relaxed solution.
\label{Existenceofweaksolutions2021}
\end{lemma}

The existence of relaxed solutions is standard (see \cite{Briani2018, Daudin2020}) but we give the proof in Appendix \ref{Existencerelaxedsolutions} for the sake of completeness and because we will use the same line of arguments at different points in our analysis.

Notice that it would not be more difficult to obtain weak solutions directly for the constrained problem. However, for the constrained problem, we don't know how to directly compute the optimality conditions and more importantly they would not give us the regularity of the Lagrange multipliers that we get thanks to our penalization procedure.

Now we fix a solution $(\tilde{m},\tilde{\omega})$ of the penalized problem and we proceed to show that $(\tilde{m},\tilde{\omega})$ is solution to a suitable linearized problem for which it will be easier to compute the optimality conditions. In the proof of the following lemma we will use a smooth distance-like function. To this end we consider a family $(\varphi_i)_{i \in \mathbb{N}}$ of functions in $\mathcal{C}_b^2(\R^d)$ such that for $m_1,m_2 \in \mathcal{P}_2(\R^d)$ we have
$$m_1=m_2 \Leftrightarrow \forall i \in \mathbb{N} \hspace{8pt} \int_{\R^d} \varphi_i(x)d(m_1-m_2)(x)=0, $$
and we define $q: \mathcal{P}_2(\R^d) \times \mathcal{P}_2(\R^d) \rightarrow \R$ by
$$q(m_1,m_2) := \sum_{i=0}^{+\infty} \frac{\left| \displaystyle \int_{\R^d} \varphi_i d(m_1-m_2) \right|^2}{2^{i}(1 + \|\varphi_{i} \|^2_{\infty} + \|D\varphi_{i} \|^2_{\infty})}.$$
Notice that $q$ satisfies 
\begin{equation}
\left \{
\begin{array}{lr}
\displaystyle  q(m_1,m_2) \geq 0  & \forall m_1,m_2 \in \mathcal{P}_2(\R^d) \\
 \displaystyle q(m_1,m_2)=0 \mbox{ if and only if  } m_1=m_2.
 \end{array}
 \right.
 \label{Distancelikeq}
 \end{equation}
It is straightforward to verify that $q$ is $\mathcal{C}^1$ with respect to both of its arguments and that 
$$\frac{\delta q}{\delta m_1}(m_1,m_2)(x) = \sum_{i=0}^{+\infty} \frac{2 \displaystyle \int_{\R^d} \varphi_i d(m_1-m_2) }{2^{i}(1 + \|\varphi_{i} \|^2_{\infty} + \|D\varphi_{i} \|^2_{\infty})}(\varphi_i(x) - \int_{\R^d}\varphi_i dm_1).$$
In particular we have
\begin{equation}
\left \{
\begin{array}{lr}
\displaystyle  \int_{\R^d} \frac{\delta q}{\delta m_1}(m_1,m_2)(y) dm_1(y) = 0 & \forall m_1,m_2 \in \mathcal{P}_2(\R^d), \\
 \displaystyle \frac{\delta q}{\delta m_1}(m_1,m_1)(x) =0 & \forall m_1 \in \mathcal{P}_2(\R^d) \mbox{ and } \forall x \in \R^d.
 \end{array}
 \right.
 \label{Propertiesq}
 \end{equation}

\begin{lemma} Let $(\tilde{m}, \tilde{\omega})$ be a fixed solution to Problem \eqref{PenalizedProblem}. Then there exist $ \lambda \in L^{\infty}([0,T])$ and $\beta \in \R^+$ satisfying 

\begin{minipage}{0.48 \textwidth}
\begin{equation}
\lambda(t) = \left \{ 
\begin{array}{ll} 
  0  & \mbox{if    } \Psi(\tilde{m}(t)) < 0, \\
 \lambda(t) \in [0 , 1] & \mbox{if    } \Psi(\tilde{m}(t)) = 0, \\
 1 & \mbox{if    } \Psi(\tilde{m}(t)) > 0, \\
\end{array}
\right.
\label{Exclusionlambdalemma}
\end{equation}
\end{minipage}
\begin{minipage}{0.48 \textwidth}
\begin{equation}
\beta  \left \{ 
\begin{array}{ll} 
 =0  & \mbox{if    } \Psi(\tilde{m}(T)) < 0, \\
 \in [0 , 1] & \mbox{if    } \Psi(\tilde{m}(T)) = 0, \\
 =1 & \mbox{if    } \Psi(\tilde{m}(T)) > 0, \\
\end{array}
\right.
\label{Exclusionbetalemma}
\end{equation}
\end{minipage}
such that $(\tilde{m}, \tilde{\omega})$ minimizes
\begin{align*}
 J_{\epsilon, \delta }^l(\omega , m) &:= \int_0^T \int_{\R^d}  L\Bigl(x, \frac{d \omega}{dt \otimes dm(t)}(t,x) \Bigr) dm(t)(x)dt \\
 &+ \int_0^T \int_{\R^d} \left[ \frac{\lambda(t)}{\epsilon}\frac{\delta \Psi}{\delta m}(\tilde{m}(t), x) +  \frac{\delta \mathcal{F}}{\delta m}( \tilde{m}(t),x)  \right]dm(t)(x)dt  \\
 &+ \int_{\R^d} \left[ \frac{\beta}{\delta} \frac{\delta \Psi}{\delta m}(\tilde{m}(T),x) +\frac{\delta \mathcal{G}}{\delta m}(\tilde{m}(T),x), \right] dm(T)(x)
\end{align*} 
over the pairs $(m,\omega)$ satisfying \eqref{CandidateRelaxed}. Once again, we set $J_{\epsilon,\delta}^l (m,\omega)= +\infty$ if $\omega$ is not absolutely continuous with respect to $dt \otimes m(t)$.
\label{Linearizationlemma}
\end{lemma}

\begin{proof}

To avoid uniqueness issues we add an additional cost to $J_{\epsilon,\delta}$ so that the new problem reads
\begin{equation}
\inf \left[ J_{\epsilon, \delta} (m, \omega) + \int_0^T q(m(t), \tilde{m}(t))dt \right].
\label{PbUniqueness}
\end{equation}
If $(\tilde{m}',\tilde{\omega}')$ is a solution of the above problem, then $\tilde{m}' = \tilde{m}$. This is a direct consequence of \eqref{Distancelikeq} and the fact that $(\tilde{m}, \tilde{\omega})$ is a solution of the penalized problem. We use this function $q$ (and not the Wasserstein distance for instance) because it is smooth and therefore we can differentiate it to get optimality conditions and also because $\displaystyle \frac{\delta q}{\delta m}(\tilde{m}, \tilde{m},x) = 0$ for all $x \in \R^d$ (see \eqref{Propertiesq}): therefore $q$ will not appear in the optimality conditions for $(\tilde{m}, \tilde{\omega})$.
Now, we introduce a suitable regularization of the function $r \mapsto \max(0,r)$. For all $h >0$, let $\gamma_h : \R \rightarrow \R^+$ be functions satisfying
\begin{equation*}
\left \{
\begin{array}{ll}
\displaystyle \gamma_h \in \mathcal{C}^2(\R), \gamma_h \geq 0, \\
\displaystyle \gamma_h(r) = \max(0,r) \mbox{    in    } \R \backslash [-h,h], \\
\displaystyle \sup _{r \in \R} |\gamma_h'(r)| \leq 1, \\
\displaystyle \sup_{r \in \R} | \gamma_h(r) - \max(0,r) | \rightarrow 0 \mbox{     as     } h \rightarrow 0. 
\end{array}
\right.
\end{equation*} 
We consider the regularized, penalized cost functionals
\begin{align*}
 J_{\epsilon,\delta,h} (m, \omega) &:= \int_0^T \int_{\R^d}  L \Bigl( x,\frac{d \omega}{dt \otimes dm(t)}(t,x) \Bigr)  dm(t)(x)dt + \int_0^T \mathcal{F}(m(t))dt +   \frac{1}{\epsilon} \int_0^T \Psi_h(m(t)) dt \\ & + \mathcal{G}(m(T)) + \frac{1}{\delta} \Psi_h(m(T))
\end{align*}
where $\Psi_h$ is defined for all $m \in \mathcal{P}_2(\R^d) $ by $\Psi_h(m) = \gamma_h(\Psi(m))$. Now we argue as in the proof of Lemma \ref{Existenceofweaksolutions2021} (see Appendix \ref{Existencerelaxedsolutions}) and find for all $h \in (0,1)$ a solution $(m_h,\omega_h)$ of 
\begin{equation}
 \inf \left[ J_{\epsilon,\delta,h}(m,\omega) + \int_0^T q(m(t), \tilde{m}(t))dt  \right].
 \label{PenalizedRegularizedModified}
 \end{equation}
Taking for granted that we can find a candidate $(\bar{m}, \bar{\omega})$ such that $J(\bar{m}, \bar{\omega}) <+\infty$ and $\Psi(\bar{m}(t))\leq 0$ for all $t\in [0,T]$ (we explicitly construct such a candidate in Lemma \ref{NewControllabilitySeptembre2022} in Section \ref{sec: Uniformestimates} below)  we find that $J_{\epsilon, \delta,h}(m_h,\omega_h)$ is bounded from above by $J(\bar{m}, \bar{\omega})$ independently of $\epsilon,\delta$ and $h$. By coercivity of $L$ we deduce that 
$$\displaystyle \int_{0}^T\int_{\R^d} \left| \frac{d\omega_h}{dt \otimes dm_h(t)}(t,x) \right|^2 dm_h(t)(x)dt \leq C$$ 
for some $C>0$ independent of $\epsilon, \delta$ and $h$. Following the proof of Lemma \ref{Existenceofweaksolutions2021} in Appendix \ref{Existencerelaxedsolutions}, we deduce that $(m_h,\omega_h)$ converges, up to a sub-sequence, in $\mathcal{C}([0,T], \mathcal{P}_r(\R^d)) \times \mathcal{M}([0,T] \times \R^d , \R^d)$ for some $r \in (1,2)$ to an element $(m',\omega')$ of $\mathcal{C}([0,T], \mathcal{P}_2(\R^d)) \times \mathcal{M}([0,T] \times \R^d , \R^d)$ satisfying \eqref{CandidateRelaxed} with $\omega'$ absolutely continuous with respect to $m'$. Let us prove that $(m',\omega')$ is a minimizer of \eqref{PbUniqueness} and therefore, by uniqueness --that is why we added the $q$-term in the cost functional--, $m'= \tilde{m}$. We just need to show that 
$$J_{\epsilon, \delta}(m', \omega') + \int_0^T q(m'(t),\tilde{m}(t))dt \leq J_{\epsilon, \delta} (\tilde{m}, \tilde{\omega}).$$ 
However, for any $h \in (0,1)$, using the minimality of $(m_h, \omega_h)$ for Problem \eqref{PenalizedRegularizedModified} it holds,
\begin{align*}
& J_{\epsilon, \delta} (m', \omega') + \int_0^T q(m'(t), \tilde{m}(t)) dt - J_{\epsilon, \delta}(\tilde{m}, \tilde{\omega}) \\
&= J_{\epsilon, \delta, h} (m_h, \omega_h) + \int_0^T q(m_h(t), \tilde{m}(t))dt - J_{\epsilon, \delta, h} (\tilde{m}, \tilde{\omega}) \\
&+ J_{\epsilon, \delta} (m',\omega') - J_{\epsilon,\delta,h}(m_h,\omega_h) + \int_0^T q(m'(t), \tilde{m}(t))dt - \int_0^T q(m_h(t), \tilde{m}(t))dt\\
&+ J_{\epsilon,\delta,h}(\tilde{m}, \tilde{\omega}) - J_{\epsilon,\delta}(\tilde{m}, \tilde{\omega}) \\
& \leq J_{\epsilon, \delta} (m', \omega') - J_{\epsilon,\delta,h}(m_h, \omega_h) + \int_0^T q(m'(t), \tilde{m}(t))dt - \int_0^T q(m_h(t), \tilde{m}(t))dt \\
&+ J_{\epsilon, \delta, h}(\tilde{m}, \tilde{\omega}) - J_{\epsilon, \delta}(\tilde{m}, \tilde{\omega}).
\end{align*}
Since $\displaystyle  \int_0^T q(m'(t), \tilde{m}(t))dt - \int_0^T q(m_h(t), \tilde{m}(t))dt$ and $\displaystyle  J_{\epsilon, \delta, h}(\tilde{m}, \tilde{\omega}) - J_{\epsilon, \delta}(\tilde{m}, \tilde{\omega})$ converge to $0$ as $h$ converges to $0$, it is sufficient to prove that $\displaystyle J_{\epsilon, \delta}(m', \omega') \leq \liminf_{h \rightarrow 0} J_{\epsilon, \delta,h}(m_h,\omega_h)$. For all $h >0$ we can rewrite
$$ J_{\epsilon, \delta, h}(m_h, \omega_h) = J_{\epsilon, \delta}(m_h, \omega_h) + \frac{1}{\epsilon} \int_0^T \left[ \Psi_h(m_h(t)) - \Psi^+(m_h(t)) \right] dt + \frac{1}{\delta} \left[ \Psi(m_h(T)) - \Psi^+(m_h(T)) \right] $$
but $$ \lim_{h \rightarrow 0}  \frac{1}{\epsilon} \int_0^T \left[ \Psi_h(m_h(t)) - \Psi^+(m_h(t)) \right] dt + \frac{1}{\delta} \left[ \Psi_h(m_h(T)) - \Psi^+(m_h(T)) \right] = 0$$
and therefore $\liminf_{h \rightarrow 0} J_{\epsilon, \delta,h}(m_h, \omega_h) = \liminf_{h \rightarrow 0} J_{\epsilon,\delta}(m_h, \omega_h)$. Finally we can conclude  by lower semi-continuity of $J_{\epsilon, \delta}$ that $\liminf_{h \rightarrow 0} J_{\epsilon,\delta}(m_h, \omega_h) \leq J_{\epsilon,\delta}(m', \omega')$. The lower semi-continuity of $J_{\epsilon,\delta}$ can be proved following Theorem 2.34 of \cite{Ambrosio2000}.

Now we argue as in \cite{Daudin2020} Section 4.1 to show that, for all $h >0$, $(m_h, \omega_h)$ is actually an infimum of the linearized problem
\begin{equation}
\displaystyle \inf J_{\epsilon, \delta,h}^l(m,\omega) + \int_0^T \int_{\R^d} \frac{\delta q}{\delta m_1}(m_h(t),\tilde{m}(t),x) dm(t)(x)dt
\label{PenalizedRegularizedLinearized}
\end{equation}
where the infimum is still taken over relaxed candidates $(m,\omega)$ satisfying \eqref{CandidateRelaxed}  with the linearized cost functional $J_{\epsilon, \delta, h}^l$ defined by
\begin{align*}
 J_{\epsilon, \delta,h }^l(\omega , m) &= \int_0^T \int_{\R^d} L \Bigl(x, \frac{d \omega}{dt \otimes dm(t)}(t,x) \Bigr) dm(t)(x)dt \\
 &+ \int_0^T \int_{\R^d} \left[ \frac{1}{\epsilon}\frac{\delta \Psi_h}{\delta m}(m_h(t), x) +  \frac{\delta \mathcal{F}}{\delta m}( m_h(t),x) \right]dm(t)(x)dt  \\
 &+ \int_{\R^d} \left[ \frac{1}{\delta} \frac{\delta \Psi_h}{\delta m}(m_h(T),x) + \frac{\delta \mathcal{G}}{\delta m}(m_h(T),x) \right] dm(T)(x),
\end{align*} 
with, once again $J_{\epsilon,\delta,h}^l(\omega,m) = +\infty$ if $\omega$ is not absolutely continuous with respect to $m(t) \otimes dt$.
 
Indeed, take a candidate $(m, \omega)$ with finite cost, take $r \in (0,1)$ and define $(m_r, \omega_r) := (1-r)(m_h, \omega_h) + r (m,\omega)$. By minimality of $(m_h, \omega_h)$ we have, for all $r \in (0,1)$
$$ \frac{1}{r} \left[ J_{\epsilon, \delta, h}(m_h,\omega_h) + \int_0^T q(m_h(t), \tilde{m}(t))dt - J_{\epsilon, \delta, h}(m_r,\omega_r) - \int_0^T q(m_r(t), \tilde{m}(t))dt \right] \leq 0. $$
Letting $r \rightarrow 0$ in the expression above and using, on the one hand, the convexity of $\displaystyle (m,\omega) \mapsto \int_0^T \int_{\R^d} L(x, \frac{d\omega}{dt\otimes dm(t)}(t,x))dm(t)(x)$ and, on the other hand, the differentiability of the mean-field costs, we show that $(m_h,\omega_h)$ is indeed a minimum of \eqref{PenalizedRegularizedLinearized}.

Now we are going to pass to the limit in the linearized problems when $h \rightarrow 0$.

On the one hand, being the family of functions $t \mapsto \gamma_h' (\Psi(m_h(t))$ bounded in $L^{\infty}([0,T])$, it converges --up to a sub-sequence-- for the weak-$*$ topology $\sigma(L^{\infty},L^1)$ of $L^{\infty}([0,T])$ to a function $\lambda$ in $L^{\infty}([0,T])$. It is easily seen that $\lambda$ satisfies \eqref{Exclusionlambdalemma}. On the other hand the functions $\displaystyle t \mapsto \int_{\R^d} \frac{\delta \Psi}{\delta m}(m_h(t),x)dm(t)(x)$ converge uniformly to $ \displaystyle t \mapsto \int_{\R^d} \frac{\delta \Psi}{\delta m}(\tilde{m}(t),x)dm(t)(x)$ as $h$ goes to $0$. Therefore we can conclude that, up to a sub-sequence, 
\begin{align*} 
\int_0^T \int_{\R^d}\frac{\delta \Psi_h}{\delta m}(m_h(t),x)dm(t)(x)dt &= \int_0^T \gamma_h'(\Psi(m_h(t)) \int_{\R^d} \frac{\delta \Psi}{\delta m}(m_h(t),x)dm(t)(x) dt \\
& \rightarrow \int_0^T \lambda(t) \int_{\R^d} \frac{\delta \Psi}{\delta m}(\tilde{m}(t),x)dm(t)(x) dt
\end{align*} 
 as $h$ goes to $0$. A similar statement holds for $\displaystyle \frac{1}{\delta} \int_{\R^d} \frac{\delta \Psi_h}{\delta m}(m_h(T),x)dm(T)(x)$ and we can conclude that, up to a sub-sequence, $J_{\epsilon,\delta,h}^l(m,\omega)$ converges to $J_{\epsilon, \delta}^l(m,\omega)$ for any relaxed candidate $(m,\omega)$, where $\displaystyle J_{\epsilon,\delta}^l$ is defined in the statement of the lemma for some $\lambda, \beta$ satisfying the conditions \eqref{Exclusionlambdalemma} and \eqref{Exclusionbetalemma}. We deduce that $(\tilde{m}, \omega')$ is an infimum of $\displaystyle J_{\epsilon,\delta}^l$. 
 Notice that the term involving $\displaystyle \frac{\delta q}{\delta m_1}$ in \eqref{PenalizedRegularizedLinearized} disappeared since  $ \displaystyle  \frac{\delta q}{\delta m_1}( \tilde{m}(t), \tilde{m}(t),x) = 0$ for all $x \in \R^d$.
To conclude that $(\tilde{m}, \tilde{\omega})$ is a solution to the linearized problem, it suffices to notice that, $(\tilde{m}, \tilde{\omega})$ being a solution to the penalized problem it must hold that 
 $$ \int_0^T \int_{\R^d} L \Bigl(x, \frac{d \tilde{\omega}}{dt \otimes d\tilde{m}(t)}(t,x) \Bigr) d\tilde{m}(t)(x)dt \leq \int_0^T \int_{\R^d} L \Bigl(x, \frac{d \omega'}{dt \otimes d\tilde{m}(t)}(t,x) \Bigr) d\tilde{m}(t)(x)dt $$
 (all the other terms in the $J_{\epsilon,\delta}$ only involve $\tilde{m}$) and therefore $J_{\epsilon, \delta}^l(\tilde{m},\tilde{\omega}) \leq J_{\epsilon, \delta}^l(\tilde{m}, \omega')$. This concludes the proof of the lemma.


\end{proof}

Before we can prove Theorem \ref{TheoremPen} we need the following duality formula.

\begin{lemma}
Assume that $(m,\alpha) \in \mathcal{C}([0,T],\mathcal{P}_2(\R^d)) \times L^2_{dt \otimes dm(t)}([0,T] \times \R^d,\R^d)$ solves the Fokker-Planck equation \eqref{FP} in the sense of distributions. Assume that $u \in \mathcal{C}([0,T],E_n)$ is a solution to the HJB equation \eqref{HJB12septembre2022} in the sense of Definition \ref{HJB12septembre2022} with inputs $(f,g) \in L^1([0,T],E_n) \times E_{n+\alpha}$. Then, for all $t_1,t_2 \in [0,T]$ it holds
\begin{align}
\notag \int_{\R^d} u(t_2,x)dm(t_2)(x) &= \int_{\R^d} u(t_1,x)dm(t_1)(x) - \int_{t_1}^{t_2} \int_{\R^d} f(t,x) dm(t)(x)dt\\
&+ \int_{t_1}^{t_2} \int_{\R^d} \left[ H(x,Du(t,x)) +  \alpha(t,x).Du(t,x) \right] dm(t)(x)dt.
\label{Duality14Septembre2022}
\end{align}

\label{DualityLemmaSeptembre2022}
\end{lemma}

\begin{proof}
We take a sequence of functions $f_m \in \mathcal{C}([0,T],E_n)$ converging to $f$ in $L^1([0,T],E_n)$ and we let $u_m$ be the corresponding solutions to the HJB equation with data $(f_m,g)$. Being $f_m$ in $\mathcal{C}([0,T],E_n)$, it is straightforward from the definition of solution \ref{HJB12septembre2022} that $u_m$ is differentiable in time, $\partial_t u_m$ belongs to $L^{\infty}([0,T],E_{n-2})$ and the HJB equation is satisfied in the strong sense. The curve $m(t)$ being bounded in $\mathcal{P}_2(\R^d)$, an approximation argument similar to \cite{Trevisan2016} Remark 2.3 shows that the integration by part formula \eqref{IPPDefinitionFokkerPlanck} holds for $u_m$ and therefore, we get
\begin{align*}
\int_{\R^d} u_m(t_2,x) dm(t_2)(x) &- \int_{\R^d} u_m(t_1,x) dm(t_1)(x) \\
& =\int_{t_1}^{t_2} \int_{\R^d} \left[ \partial_t u_m(t,x) +\alpha(t,x).Du_m(t,x) +\Delta u_m(t,x) \right ]dm(t)(x)dt \\
& =\int_{t_1}^{t_2} \int_{\R^d} \left[ \alpha(t,x).Du_m(t,x) +H(x,Du_m(t,x)) -f_m(t,x)\right ]dm(t)(x)dt 
\end{align*}
where we used the equation satisfied by $u_m$ at the last line. Now we can use the stability result of Theorem \ref{TheoremHJB14Septembre2022} to pass to the limit as $m\rightarrow +\infty$ and conclude the proof of the proposition.
\end{proof}

Finally we can conclude the proof of Theorem \ref{TheoremPen}. 

\begin{proof}[Proof of Theorem \ref{TheoremPen}]
We consider $\tilde{u} \in \mathcal{C}([0,T], E_n )$ solution to 
\begin{equation}
\left\{
\begin{array}{ll}
\displaystyle -\partial_t \tilde{u}(t,x) + H(x, D\tilde{u}(t,x)) -\Delta \tilde{u}(t,x) \\
\displaystyle \hspace{80pt}  = \frac{\lambda(t)}{\epsilon} \frac{\delta \Psi}{\delta m}(\tilde{m}(t),x)  +  \frac{\delta \mathcal{F}}{\delta m}(\tilde{m}(t),x) \hspace{30pt} \mbox{ in } (0,T) \times \R^d,\\
\displaystyle \tilde{u}(T,x) = \frac{\beta}{\delta} \frac{\delta \Psi}{\delta m}(\tilde{m}(T),x) + \frac{\delta \mathcal{G}}{\delta m}(\tilde{m}(T),x) \hspace{30pt} \mbox{ in } \R^d,
\end{array}
\right.
\label{HJBbackward}
\end{equation}
---the existence of such a solution is guaranteed by Theorem \ref{TheoremHJB14Septembre2022}---
and we proceed by verification. We use Lemma \ref{DualityLemmaSeptembre2022}  to get
\begin{equation*}
\int_{\R^d} \tilde{u}(0,x) dm_0(x)  = - \int_0^T \int_{\R^d} \left[ H(x,D \tilde{u}(t,x) )+ \frac{d \tilde{\omega}}{dt \otimes d\tilde{m}}(t,x) .D \tilde{u}(t,x)  \right] d\tilde{m}(t)dt.
\end{equation*}
Here we used the equation satisfied by $\tilde{u}$ and the convention $\displaystyle \int_{\R^d} \frac{\delta U}{\delta m}(m,x)dm(x) = 0$ for all $m \in \mathcal{P}_2(\R^d)$ and all $\mathcal{C}^1$ map $U$. But the inequality
$$ -H(x,D\tilde{u}(t,x) ) - \frac{d \tilde{\omega} }{dt \otimes d\tilde{m}(t)} (t,x) .D\tilde{u}(t,x)  \leq L(x, \frac{ d\tilde{\omega} } {dt \otimes d \tilde{m} } (t,x)) $$
holds, with equality if and only if 
$$ \frac{ d\tilde{\omega} } {dt \otimes d \tilde{m} } (t,x)  = -D_pH(x, D\tilde{u}(t,x) ). $$
Therefore, 
$$\displaystyle \int_{\R^d} \tilde{u}(0,x) dm_0(x) \leq J_{\epsilon, \delta}^l(\tilde{m}, \tilde{\omega})$$
with equality if and only if $ \displaystyle \frac{ d\tilde{\omega} } {dt \otimes d \tilde{m} } (t,x)  = -D_pH(x, D\tilde{u}(t,x) )$, $dt \otimes \tilde{m}(t)$-almost everywhere. Now if we consider $ \tilde{m}'$ solution to 
$$  \partial_t \tilde{m}' - \mdiv( D_pH(x,D \tilde{u}(t,x) ) \tilde{m}' )  - \Delta \tilde{m}' =0 $$
with $\tilde{m}'(0) = m_0$, a similar computation shows that 
$$ \int_{\R^d} \tilde{u}(0,x) dm_0(x) =  J_{\epsilon, \delta}^l ( -D_pH(x,D \tilde{u}(t,x) ) \tilde{m}' , \tilde{m}' ) $$
which means that the cost $\displaystyle \int_{\R^d} \tilde{u}(0,x) dm_0(x)$ can indeed be reached and, by minamility of $(\tilde{\omega}, \tilde{m} )$ we get
\begin{equation} 
\int_{\R^d} \tilde{u}(0,x) dm_0(x)  = \inf_{(\omega, m) }  J_{\epsilon, \delta}^l
\label{valuelinearized}
\end{equation} 
and 
$$\tilde{\omega} = -D_pH(x, D\tilde{u}(t,x) ) \tilde{m}(t) \otimes dt. $$
Combining the Fokker-Planck equation in \eqref{CandidateRelaxed} where $\tilde{\omega}$ is replaced by $-D_pH(x, D\tilde{u}(t,x) ) \tilde{m}(t) \otimes dt$ with the HJB equation \eqref{HJBbackward} and recalling that $\lambda$ and $\beta$ satisfy the conditions of Lemma \ref{Linearizationlemma} concludes the proof of the theorem.
\end{proof}

\section{From the penalized problems to the constrained one}

\label{sec: Proof of the main theorem}

The first goal of this section is to find estimates on the system of optimality conditions \eqref{OCPen} which are independent from $\epsilon$ and $\delta$. This is Section \ref{sec: Uniformestimates}. Next we prove the regularity and find suitable expressions for the first two derivatives of the map $t \mapsto \Psi(m(t))$ when $(m,\alpha)$ is a solution to the penalized problem. This is Section \ref{sec: Second order analysis}. Finally we prove Theorems \ref{MainTheoremStayInside} and \ref{MainTheorem2021} in Section \ref{sec: Proof of the Main Theorems}.

\subsection{Uniform (in epsilon, delta) estimates}

\label{sec: Uniformestimates}

First we construct a candidate $(\bar{m}, \bar{\alpha})$ which stays uniformly inside the constraint at all time with a finite cost.

\begin{lemma} 
Provided $\Psi(m_0) <0$, we can build a trajectory $(\bar{m},\bar{\alpha})$ in $\mathcal{C}([0,T], \mathcal{P}_2(\R^d)) \times L^2_{dt \otimes m(t)}([0,T] \times \R^d , \R^d)$ such that $J(\bar{\alpha}, \bar{m}) < +\infty$ and $\Psi(\bar{m}(t) ) \leq -\theta$ for all $t$ in $[0,T]$, for some $\theta >0$.
\label{NewControllabilitySeptembre2022}
\end{lemma}

\begin{proof}
First we introduce a probability space $(\Omega, \mathcal{F}, \mathbb{P})$ supporting a random variable $X_0$ with law $m_0$ and an independent Brownian motion $(B_t)$. Take $c>0$ and consider a solution to the SDE
$$dX_t = -c(X_t-X_0) dt + \sqrt{2}dB_t, \hspace{30pt} X|_{t=0} = X_0. $$
A simple application of Itô's lemma proves that $X_t$ can be rewritten as
\begin{equation}
X_t = X_0 + \sqrt{2}\int_0^te^{-c(t-s)}dB_s 
\label{reecritureXtSeptembre2022}
\end{equation}
and therefore 
\begin{equation*}
\E\left[ |X_t- X_0 |^2 \right]  = 2 \int_0^te^{-2c(t-s)}ds = \frac{1}{c}(1-e^{-2ct}).
\end{equation*}
Now let $\bar{m}(t)$ be the law of $X_t$. The above computation shows that 
$$d_2^2(\bar{m}(t),m_0) \leq \frac{1}{c}, \hspace{30pt} \forall t\in [0,T].$$
With an abstract mimicking argument as in \cite{Brunick2013} we can find a measurable drift $\bar{\alpha} : [0,T] \times \R^d \rightarrow \R^d$ such that 
$$ \partial_t \bar{m}+\mdiv(\bar{\alpha} \bar{m} ) - \Delta \bar{m} = 0$$
and 
\begin{equation*} 
\int_0^T \int_{\R^d} |\bar{\alpha}(t,x) |^2d\bar{m}(t)(x)dt \leq \mathbb{E} \left[ \int_0^T c^2 \left| X_t - X_0 \right |^2 dt \right]\leq cT.
\end{equation*}
However a direct computation, using Jensen's inequality, shows that it is enough to take, for all $(t,x) \in (0,T] \times \R^d$, 
$$ \bar{\alpha}(t,x) := \frac{c}{\bar{m}(t,x)} \int_{\R^d} (x-y)m^{y}(t,x)dm_0(y)$$
where $m^y(t)$ is the solution to 
\begin{equation*}
\left \{
\begin{array}{ll}
\partial_t m^y -c\mdiv( (x-y)m^y) - \Delta m^y =0 \\
m^y(0)=\delta_y.
\end{array}
\right.
\end{equation*}
Notice that $X_0$ being independent from the Brownian motion, we easily deduce from \eqref{reecritureXtSeptembre2022} that $\bar{m}(t,x) >0$ for all $(t,x) \in (0,T] \times \R^d$.

Being $\Psi$ Lipschitz continuous and $\Psi(m_0) <0$ we can choose $c$ large enough so that $\Psi(\bar{m}(t)) \leq \frac{\Psi(m_0)}{2}$ for all $t \in [0,T]$ and this concludes the proof of the lemma.
\end{proof}

Using this particular candidate and the convexity of the constraint we can obtain the following estimate which is crucial to find compactness in the problem. 

Although the notations do not make it clear, from now on $(m,u,\lambda,\beta)$ will generally denote a solution to the optimality conditions \eqref{OCPen} for the penalized problem \eqref{PenalizedProblem} and therefore depend upon a particular $(\epsilon,\delta)$.
\begin{lemma} There is a constant $C=C(\Psi(m_0))>0$ such that, for all $\epsilon, \delta >0$ and for all tuple $(u,m, \lambda,\beta)$ satisfying the conditions of Theorem \ref{TheoremPen} it holds
$$ \displaystyle \frac{1}{\epsilon} \int_0^T \lambda(t) dt + \frac{\beta}{\delta} \leq C. $$
\label{LemmaL1Estimate}
\end{lemma}
\begin{proof}
By Lemma \ref{NewControllabilitySeptembre2022} we can build a solution of the Fokker-Planck equation $(\bar{\alpha}, \bar{m})$ such that $J(\bar{\alpha}, \bar{m}) < +\infty $ and, for all $t \in [0,T]$, $\Psi( \bar{m}(t)) \leq - \theta $ for some $\theta >0$ independent of $t$. Using the fact that $(\bar{m}, \bar{\alpha})$ solves the Fokker-Planck equation, we can apply Lemma \ref{DualityLemmaSeptembre2022} to get
\begin{align*}
\int_0^T \int_{\R^d} &\left[ \bar{\alpha}(t,x).Du(t,x) + H(x,Du(t,x)) - \frac{\lambda(t)}{\epsilon} \frac{\delta \Psi }{\delta m}(m(t),x) -\frac{\delta \mathcal{F}}{\delta m}(m(t),x)\right] d \bar{m}(t)(x)dt \\
&= \int_{\R^d} \left[ \frac{\beta}{\delta} \frac{\delta \Psi}{\delta m}(m(T),x) +\frac{\delta \mathcal{G}}{\delta m}(m(T),x) \right]d \bar{m}(T)(x) - \int_{\R^d} u(0,x)dm_0(x).
\end{align*}
Now, reorganizing the terms and using the fact that, by definition of $L$, we have for all $(t,x)$ in $[0,T] \times \R^d$
$$ \bar{ \alpha} (t,x).Du(t,x) +  H(x,Du(t,x)) \geq -L(x, \bar{\alpha}(t,x)), $$
we get
\begin{align}
\notag -\int_0^T \int_{\R^d} & \frac{\lambda(t)}{\epsilon} \frac{\delta \Psi}{\delta m}(m(t),x) d \bar{m}(t)(x)dt - \int_{\R^d} \frac{\beta}{\delta} \frac{\delta \Psi}{\delta m}(m(T),x) d \bar{m}(T)(x) \\
\notag & \leq \int_0^T \int_{\R^d} \left[ L(x, \bar{\alpha}(t,x))  + \frac{\delta \mathcal{F}}{\delta m}(m(t),x) \right] d \bar{m}(t)(x)dt \\
 &+ \int_{\R^d} \frac{\delta \mathcal{G}}{\delta m}(m(T),x) d\bar{m}(T)(x) - \int_{\R^d} u(0,x) dm_0(x).
\label{InegaliteLemmeEstimationControllabilite}
\end{align}
On the one hand -using \eqref{valuelinearized} in the proof of Theorem \ref{TheoremPen} and the notations therein- we have that $\displaystyle \int_{\R^d} u(0,x)dm_0(x) = J_l^{\epsilon, \delta}(\tilde{m}, \tilde{\omega} )$. But the linearized costs cancel out when applied to $(\tilde{m}, \tilde{\omega})$ and therefore $J_l^{\epsilon, \delta}(\tilde{m}, \tilde{\omega} ) = J(\tilde{m}, \tilde{\omega})$. And since $L$, $\mathcal{F}$ and $\mathcal{G}$ are bounded from below we get a lower bound on $\displaystyle \int_{\R^d} u(0,x)dm_0(x)$ independent of $\epsilon$ and $\delta$. The other terms in the right-hand side of \eqref{InegaliteLemmeEstimationControllabilite} are also bounded from above since $J(\bar{\alpha},\bar{m}) <+\infty$ and since $x \mapsto \displaystyle \frac{\delta \mathcal{F}}{\delta m}(m,x)$ and $x\mapsto \displaystyle \frac{\delta \mathcal{G}}{\delta m}(m,x)$ are bounded in $E_n$ with bounds uniform in $m$ and $\bar{m}(t)$ belongs to $\mathcal{P}_2(\R^d)$ for all $t\in [0,T]$ . On the other hand, by convexity of $\Psi$ we get for all $t \in [0,T]$,
\begin{align*}
\int_{\R^d} \frac{\delta \Psi}{\delta m} (m(t),x) d \bar{m}(t)(x) & \leq \Psi( \bar{m}(t)) - \Psi(m(t)) \\
& \leq -\theta - \Psi(m(t))
\end{align*}
and by definition of $\lambda$ and $\beta$ we have $ \displaystyle \lambda(t) \Psi(m(t)) \geq 0 $  for all $t \in [0,T]$ and $\beta \Psi(m(T)) \geq 0$ and thus, if $C>0$ is an upper bound for the right-hand side of \eqref{InegaliteLemmeEstimationControllabilite} we get
$$ \int_0^T \frac{\lambda(t)}{\epsilon} dt + \frac{\beta}{\delta} \leq \frac{C}{\theta}, $$
which concludes the proof of the Lemma.
\end{proof}

\begin{remark}
Notice that this estimate, together with the  construction of Lemma \eqref{NewControllabilitySeptembre2022} are the only steps which require the convexity of $\Psi$, Assumption \eqref{PsiConvex} as well as the condition that $\Psi(m_0)$ must be strictly negative, Assumption \eqref{SartInside}.
\end{remark}

We can combine this Lemma with Theorem \ref{TheoremHJB14Septembre2022} to find uniform in $\epsilon,\delta$ estimates for the system of Optimality Conditions \eqref{OCPen}.

\begin{proposition} There is some $C  >0$ such that, for any $\epsilon, \delta >0$ and any solution $(m,u,\lambda,\beta)$ of \eqref{OCPen} satisfying \eqref{Exclusionlambda} and \eqref{Exclusionbeta} it holds 
$$ \sup_{t\in [0,T]} \|u(t) \|_n \leq C. $$
\label{PropositionEstimates}
\end{proposition}

At this stage, the above estimates would be sufficient to pass to the limit when $\epsilon$ and $\delta$ go to zero in the penalized problem \eqref{PenalizedProblem}. We would find, at the limit, solutions of the constrained problem \eqref{Problem} and passing to the limit in the optimality conditions we would find that the solutions to the constrained problem satisfy similar conditions with $\displaystyle \frac{\lambda}{\epsilon} $ replaced by a non-negative Radon measure $\nu \in \mathcal{M}^+([0,T])$. This would lead to a priori discontinuous (in time) optimal controls. However, we refrain from following such approach for now. Instead we are going to exhibit a special behavior of the optimal solutions of the penalized problem. Indeed we are going to show in the next section that solutions of the penalized problem stay inside the constraint when the penalization is strong enough. Consequently it is sufficient to take $\epsilon$ and $\delta$ small to get solutions to the constrained problem and optimal controls for the constrained problem are still continuous.

\subsection{Second order analysis}

\label{sec: Second order analysis}

The special behavior (described just above) of the solutions will be a simple consequence of the fact that we cannot have simultaneously  $\Psi(m(t)) > 0$ and $\frac{d^2}{dt^2}\Psi(m(t)) \leq 0$ (here $m$ is a solution to \eqref{PenalizedProblem})   when the penalization is strong enough. The purpose of this section is to prove the regularity and a suitable expansion of the map $t \mapsto \Psi(m(t))$.
%
%
\begin{proposition}
Suppose that $(m,u, \lambda, \beta)$ is a solution of \eqref{OCPen} for some $\epsilon, \delta >0$. Then the map $ t \mapsto \Psi(m(t)) $ is $\mathcal{C}^1$ in $[0,T]$ and $\mathcal{C}^2$ in $[0,T] \bigcap \{t:  \Psi(m(t)) \neq 0 \}$ with derivatives given by
\begin{align*}
\frac{d}{dt}\Psi(m(t)) &= - \int_{\R^d} D_m \Psi(m(t),x).D_pH(x,Du(t,x)) dm(t)(x) \\
&+ \int_{\R^d} \mdiv_x D_m \Psi(m(t),x) dm(t)(x) 
\end{align*}
and
\begin{align*}
 \frac{d^2}{dt^2} \Psi(m(t)) &= \frac{\lambda(t)}{\epsilon} \int_{\R^d} D_m\Psi(m(t),x).D^2_{pp}H(x,Du(t,x))D_m\Psi(m(t),x)dm(t)(x) \\
 & + F(Du(t), D^2u(t),D\Delta u(t), m(t) )
 \end{align*}
for some functional $F : \mathcal{C}_b(\R^d, \R^d) \times \mathcal{C}_b(\R^d, \mathbb{S}_d(\R)) \times \mathcal{C}_b(\R^d, \R^d)  \times \mathcal{P}_2(\R^d) \rightarrow \R$ independent of $\epsilon$ and $\delta$ and bounded in sets of the form $\mathcal{A} \times \mathcal{P}_2(\R^d)$ for bounded subsets $\mathcal{A}$ of $\mathcal{C}_b(\R^d, \R^d) \times \mathcal{C}_b(\R^d, \mathbb{S}_d(\R)) \times \mathcal{C}_b(\R^d, \R^d)$. 

\label{PropDerivatives}

\end{proposition}

\begin{proof}
Since $\Psi$ is supposed to satisfy Assumption \eqref{Newregularity2022},  we can use Proposition \ref{ItoFlowMeasures} and, for all $ t\in [0,T]$ we get
\begin{align*}
\Psi(m(t)) &= \Psi(m_0) - \int_0^t \int_{\R^d} D_m\Psi(m(s),x).D_pH(x,Du(s,x))dm(s)(x)ds \\
&+ \int_0^t\int_{\R^d} \mdiv_xD_m\Psi(m(s),x)dm(s)(x)ds. 
\end{align*}
Being $u$ in $\mathcal{C}([0,T] , E_n)$ and $m$ in $\mathcal{C}([0,T] , \mathcal{P}_2(\R^d))$ we get that $t \mapsto \Psi(m(t))$ is $\mathcal{C}^1$ with
\begin{align*}
\frac{d}{dt}\Psi(m(t)) &= - \int_{\R^d} D_m \Psi(m(t),x).D_pH(x,Du(t,x)) dm(t)(x) \\
&+ \int_{\R^d} \mdiv_x D_m \Psi(m(t),x) dm(t)(x). 
\end{align*}
Now we assume that $\Psi(m(t)) \neq 0$. We denote by $v(t,x)$ the integrand
$$v(t,x) := -D_m\Psi(m(t),x).D_pH(x,Du(t,x)) + \mdiv_xD_m\Psi(m(t),x) $$
The parameter $\lambda $ is constant (equal to $0$ or $1$) in a neighborhood $(t_1,t_2)$ of $t$ because of the exclusion condition \eqref{Exclusionlambda} and $u$ solves the HJB equation according to Definition \ref{DefinitionHJB12Septembre2022} so we have that $u$ belongs to $\mathcal{C}^{1,2}((t_1,t_2) \times \R^d)$. Moreover, 
$$\displaystyle \partial_t u (t,x) = H(x,Du(t,x)) - \Delta u(t,x) - \frac{\lambda(t)}{\epsilon}\frac{\delta \Psi}{\delta m}(m(t),x) - \frac{\delta \mathcal{F}}{\delta m}(m(t),x)$$
and $u$ belongs to $\mathcal{C}([0,T], E_n)$ with $n \geq 3$. This means that $\partial_t u $ is differentiable with respect to $x$ with
\begin{align*}
-\partial_tDu(t,x) &+D_xH(x,Du(t,x)) +  D^2u(t,x)D_pH(x,Du(t,x)) - D\Delta u(t,x) \\
&= \frac{\lambda(t)}{\epsilon}D_m \Psi(m(t),x) + D_m\mathcal{F}(m(t),x).  
\end{align*}
But $m$ solves the Fokker-Planck equation, $\Psi$ satisfies Assumptions \eqref{Newregularity2022} and \eqref{PsiC2} so we can apply Proposition \ref{ItoFlowMeasures} to $D_m\Psi(m(t),x) $  and $\mdiv_xD_m \Psi(m(t),x)$ and deduce that $v$ belongs to $\mathcal{C}_b^{1,2}((t_1,t_2) \times \R^d)$ and therefore $t \mapsto \frac{d}{dt}\Psi(m(t))$ is differentiable at $t$ with
$$\frac{d^2}{dt^2}\Psi (m(t)) = \int_{\R^d} \left[ \partial_t v(t,x) -D_pH(x,Du(t,x)).Dv(t,x) + \Delta v(t,x) \right] dm(t)(x). $$    
Computing $\partial_t v$ leads to
\begin{align*}
 \partial_t v(t,x) &= - \frac{d}{dt} D_m\Psi(m(t),x).D_pH(x,Du(t,x))  + \frac{d}{dt} \mdiv_x D_m\Psi(m(t),x) \\
 & - D_m\Psi(m(t),x).D^2_{pp}H(x,Du(t,x))\partial_tDu(t,x) \\
 &= - \frac{d}{dt} D_m\Psi(m(t),x) .D_pH(x,Du(t,x)) + \frac{d}{dt} \mdiv D_m\Psi(m(t),x) \\
&- D_m \Psi(m(t),x).D^2_{pp}H(x,Du(t,x))D^2u(t,x)D_pH(x,Du(t,x)) \\
&- D_m \Psi(m(t),x).D^2_{pp}H(x,Du(t,x))D_xH(x,Du(t,x)) \\
&+  D_m \Psi(m(t),x).D^{2}_{pp}H(x,Du(t,x))D\Delta u(t,x)  \\
&+ \frac{\lambda(t)}{\epsilon} D_m \Psi(m(t),x)D^2_{pp}H(x,Du(t,x)).D_m \Psi(m(t),x) \\
&+ D_m \Psi(m(t),x).D^2_{pp}H(x,Du(t,x)) D_m\mathcal{F}(m(t),x), \\
\end{align*}
and therefore
\begin{align*}
\frac{d^2}{dt^2} \Psi(m(t)) &= \frac{\lambda(t)}{\epsilon} \int_{\R^d} D_m\Psi(m(t),x).D^2_{pp}H(x,Du(t,x))D_m\Psi(m(t),x)dm(t)(x) \\
 & + F(Du(t), D^2u(t),D\Delta u(t),m(t) )
 \end{align*}
with 
\begin{align*}
F(Du(t),& D^2u(t),D\Delta u(t), m(t)) =  \int_{\R^d}  \left[ - D_pH(x,Du(t,x)).Dv(t,x) + \Delta v(t,x) \right ] dm(t)(x) \\
& - \int_{\R^d} \frac{d}{dt} D_m\Psi(m(t),x).D_pH(x,Du(t,x)) dm(t)(x) \\
& + \int_{\R^d} \frac{d}{dt} \mdiv_x D_m\Psi(m(t),x) dm(t)(x) \\
& - \int_{\R^d} D_m \Psi(m(t),x).D^2_{pp}H(x,Du(t,x))D^2u(t,x)D_pH(x,Du(t,x)) dm(t)(x) \\
&- \int_{\R^d} D_m \Psi(m(t),x).D^2_{pp}H(x,Du(t,x))D_xH(x,Du(t,x)) dm(t)(x)\\
&+  \int_{\R^d} D_m \Psi(m(t),x).D^{2}_{pp}H(x,Du(t,x))D\Delta u(t,x) dm(t)(x) \\
&+ \int_{\R^d}  D_m \Psi(m(t),x).D^2_{pp}H(x,Du(t,x)) D_m\mathcal{F}(m(t),x) dm(t)(x).
\end{align*}
\end{proof}

\begin{remark}
An explicit formula for $Dv$, $\Delta v$ or $F$ is not necessary for our purpose however a tedious but straightforward computation leads to
\begin{align*}
\frac{d^2}{dt^2}\Psi(m(t)) &= \frac{\lambda(t)}{\epsilon} \int_{\R^d} D_m\Psi(m(t),x). D^2_{pp}H(x,Du(t,x))D_m\Psi(m(t),x)dm(t)(x) \\
&+ \int_{\R^d} \Delta_x\mdiv_xD_m\Psi(x,m(t))dm(t)(x)\\
&+ \int_{\R^d} \int_{\R^d} \mdiv_x\mdiv_yD^2_{mm}\Psi(m(t),x,y)dm(t)(x)dm(t)(y) \\
&-2 \int_{\R^d} \int_{\R^d} \mdiv_yD^2_{mm}\Psi(m(t),x,y).D_pH(x,Du(t,x))dm(t)(x)dm(t)(y) \\
&-2 \int_{\R^d} \overrightarrow{\Delta}_xD_m\Psi(m(t),x).D_pH(x,Du(t,x))dm(t)(x) \\
&+ \int_{\R^d} D_m \Psi(m(t),x).D^2_{pp}H(x,Du(t,x)) D_m \mathcal{F}(m(t),x) dm(t)(x) \\
&+ \int_{\R^d} \int_{\R^d} D^2_{mm}\Psi(m(t),x,y) D_pH(x,Du(t,x). D_pH(y,Du(t,y))dm(t)(x)dm(t)(y) \\
&+ \int_{\R^d} D_xD_m\Psi(m(t),x)D_pH(x,Du(t,x)).D_pH(x,Du(t,x))dm(t)(x) \\
&-2 \int_{\R^d} D_xD_m\Psi(m(t),x).D^2u(t,x)D^2_{pp}H(x,Du(t,x))dm(t)(x)\\
&- \sum_{i=1}^n \int_{\R^d} D_{x_i} \frac{\delta \Psi}{\delta m}(m(t),x)D^2u(t,x).D^2u(t,x) D^2_{pp}\partial_{p_i}H(x,Du(t,x))dm(t)(x) \\
&-\int_{\R^d} D_m\Psi(m(t),x). \overrightarrow{\Delta}_xD_pH(x,Du(t,x))dm(t)(x)\\
&-2\int_{\R^d} D_xD_m\Psi(m(t),x)D^2_{xp}H(x,Du(t,x))dm(t)(x) \\
&-\int_{\R^d} D_m\Psi(m(t),x).D^2_{pp}H(x,Du(t,x))D_xH(x,Du(t,x))dm(t)(x) \\
&+\int_{\R^d} D^2_{xp}H(x,Du(t,x))D_m\Psi(m(t),x).D_pH(x,Du(t,x))dm(t)(x)\\
&-2 \sum_{i=1}^n \int_{\R^d} \partial_{x_i} \frac{\delta \Psi}{\delta m}(m(t),x) D^2_{xp}\partial_{p_i}H(x,Du(t,x)).D^2u(t,x)dm(t)(x).
\end{align*}
The formula above shows in particular that the terms in $D\Delta u$ cancel out and thus $F$ depends only on the derivatives of $u$ up to order two.
\end{remark}

\subsection{Proof of the main theorems}

\label{sec: Proof of the Main Theorems}

\begin{proposition}
There is some $\epsilon_0, \delta_0>0$ such that any solution $(m, \alpha)$ of Problem \eqref{PenalizedProblem} for some $(\epsilon,\delta) \in (0,\epsilon_0] \times (0, \delta_0]$ stays inside the constraint at all time:
$$ \Psi(m(t)) \leq 0, \hspace{30pt} \forall t\in [0,T]. $$
\label{Insidetheconstraint}
\end{proposition}

\begin{proof}
The proof follows closely the methodology of \cite{Cannarsa2018a} Lemma 3.7. Toward a contradiction we suppose that there exist a sequence $(\epsilon_k,\delta_k)_{k \in \mathbb{N}} \in( (0,1)\times (0,1) )^{\mathbb{N}}$ converging to $(0,0)$, corresponding solutions $(m_k, -D_pH(x,Du_k(t,x)))_{k \in \mathbb{N}}$ satisfying the conditions of Theorem \ref{TheoremPen} with corresponding multipliers $(\lambda_k,\beta_k)$ and times $(t_k)_{k \in \mathbb{N}} \in (0,T]$ which are local maximum points of $ t \mapsto \Psi(m_k(t))$ and such that $\Psi(m_k(t_k)) >0$. The couples $(m_k,\omega_k)$ are uniformly bounded in $\mathcal{C}^{1/2}([0,T], \mathcal{P}_2(\R^d)) \times \mathcal{M}([0,T] \times \R^d,\R^d)$ and we can assume that they converge in $\mathcal{C}^{(1-\delta)/2}([0,T], \mathcal{P}_{2-\delta}(\R^d)) \times \mathcal{M}([0,T] \times \R^d, \R^d)$, for some $\delta \in (0,1)$, toward some solution $(\tilde{m},\tilde{\omega})$ to the constrained problem. In particular, $\Psi( \tilde{m}(t)) \leq 0$ for all $t\in [0,T]$.

We first notice that, thanks to Lemma \ref{LemmaL1Estimate}, for large enough $k$, $\beta_k <1$ and therefore $\Psi(m_k(T)) \leq 0$ and $t_k \neq T$.

Using Proposition  \ref{PropDerivatives} yields that $t \mapsto \Psi(m_k(t))$ is $C^2$ in a neighborhood of $t_k$ and, 
\begin{align*}
\frac{d^2}{dt^2} \Psi(m_k(t))|_{t=t_k} &= \frac{1}{\epsilon_k} \int_{\R^d} D_m\Psi(m_k(t_k),x).D^2_{pp}H(x,Du_k(t_k,x))D_m\Psi(m_k(t_k),x)dm_k(t_k)(x) \\
 & + F(Du_k(t_k), D^2u_k(t_k), D\Delta u_k(t_k), m_k(t_k) ) \\
 & \geq \frac{1}{\mu  \epsilon_k} \int_{\R^d}  |D_m\Psi(m_k(t_k),x) |^2dm_k(t_k)(x)  \\
 &+ F(Du_k(t_k), D^2u_k(t_k), D\Delta u_k(t_k), m_k(t_k) ),
 \end{align*}
where we used the strict convexity of $H$ with respect to the $p$ variable as stated in Assumption \eqref{AHallinone}. On the one hand, using the estimates of Proposition \ref{PropositionEstimates} we have that $F(Du_k(t), D^2u_k(t), D\Delta u_k(t),m_k(t) )$ is bounded independently from $k$. On the other hand, using the regularity assumption \eqref{PsiC2} and up to taking a subsequence we can assume that 
$$ \lim_{k \rightarrow + \infty} \int_{\R^d} |D_m\Psi(m_k(t_k),x) |^2 dm_k(t_k)(x)  =  \int_{\R^d} |D_m\Psi(\tilde{m}(\tilde{t}),x) |^2 d\tilde{m}(\tilde{t})(x)$$
for some $\tilde{t} \in [0,T]$ such that $\Psi(\tilde{m}(\tilde{t})) =0$. This is where Assumption \eqref{TransCondPsi} comes into play. Since $\Psi(\tilde{m}(\tilde{t})) =0$, we have that 
$$ \int_{\R^d} |D_m\Psi(\tilde{m}(\tilde{t}),x) |^2 d\tilde{m}(\tilde{t})(x) >0,$$
and we deduce that, $\displaystyle \frac{d^2}{dt^2} \Psi(m_k(t))|_{t=t_k}  >0$ for $k$ large enough. This leads to a contradiction since $t_k$ is assumed to be a local maximum point of $t \rightarrow \Psi(m_k(t))$.
\end{proof}

Theorem \ref{MainTheoremStayInside} is a direct consequence of the above proposition.

\begin{proof}[Proof of Theorem \ref{MainTheoremStayInside}]
Denote by $\bar{\mathcal{U}}_{\epsilon,\delta}$ the value of Problem \eqref{PenalizedProblem} and by $\bar{\mathcal{U}}$ the value of the constrained problem \eqref{Problem}. We assume that $(\epsilon,\delta)$ belongs to $(0,\epsilon_0) \times (0, \delta_0)$ with $(\epsilon_0, \delta_0)$ the parameters from Proposition \ref{Insidetheconstraint}.

We have that $\bar{\mathcal{U}}_{\epsilon,\delta}= \bar{\mathcal{U}}$ and the minimizers for problems \eqref{PenalizedProblem} and \eqref{Problem} coincide. 

Indeed, it is straightforward that $\bar{\mathcal{U}}_{\epsilon,\delta} \leq \bar{\mathcal{U}}$. Now if $(m_1,\alpha_1)$ is a solution to Problem \eqref{PenalizedProblem}, by Proposition \ref{Insidetheconstraint}, $(m_1,\alpha_1)$ is admissible for Problem \eqref{Problem}. This means that $\bar{\mathcal{U}}_{\epsilon,\delta} = J_{\epsilon,\delta}(m_1, \alpha_1) = J(m_1, \alpha_1) \geq \bar{\mathcal{U}}$ and, therefore $\bar{\mathcal{U}}_{\epsilon,\delta}= \bar{\mathcal{U}}$ and $(m_1,\alpha_1)$ is a solution to \eqref{Problem}. Conversely, if $(m_2, \alpha_2)$ is a solution to \eqref{Problem} then $J_{\epsilon,\delta}(m_2, \alpha_2) = J(m_2,\alpha_2) = \bar{\mathcal{U}} = \bar{\mathcal{U}}_{\epsilon,\delta}$ and $(m_2, \alpha_2)$ is a solution to \eqref{PenalizedProblem}.

Looking carefully at the proof of Proposition \ref{Insidetheconstraint}, using Theorem \eqref{TheoremHJB14Septembre2022} with the estimates given by Proposition  \ref{PropositionEstimates} and Lemma \ref{NewControllabilitySeptembre2022} we see that the threshold $(\epsilon_0,\delta_0)$ depends on $m_0$ only through the value $\Psi(m_0)$.
\end{proof}

Now we are finally able to conclude the proof of Theorem \ref{MainTheorem2021}.

\begin{proof}[Proof of Theorem \ref{MainTheorem2021}]
We use Theorem \ref{MainTheoremStayInside} and the optimality conditions for the penalized problem: If $(m,\alpha)$ is any solution to Problem \eqref{Problem}, we can find $(\epsilon, \delta) \in (0,\epsilon_0) \times (0,\delta_0) $, $\lambda \in L^{\infty}([0,T])$, $\beta \geq 0$, $u \in \mathcal{C}([0,T], \mathcal{C}_b^n(\R^d))$ such that $\alpha(t,x) = -D_pH(x,Du(t,x))$ for all $(t,x) \in [0,T] \times \R^d$ and $(m, u , \lambda, \beta)$ satisfies the conditions of Theorem \ref{TheoremPen}. Taking $\nu(t) := \frac{\lambda(t)}{\epsilon}$ and $\eta:= \frac{\beta}{\delta}$ concludes the proof of the first part of the theorem. 

Now, if we suppose that $\mathcal{F}$ and $\mathcal{G}$ are convex in the measure variable we can proceed as in \cite{Daudin2020} Section 4.3 and easily show that the conditions are sufficient.
\end{proof}

\section{The general case}

\label{sec: The general case}

The goal of this section is to prove Theorem \ref{MainTheoremWithoutTransCondPsi} . We first need to extend the results of Theorem \ref{TheoremHJB14Septembre2022} to HJB equations with right hand-side of the form $\nu \psi_1 + \varphi_1$ where $\nu$ belongs to $\mathcal{M}^+([0,T])$ and $\psi_1, \varphi_1$ belong to $\mathcal{C}([0,T], E_n)$.

\subsection{The HJB equation}

\begin{definition}
Suppose that $n \geq 3$. Let $\psi_1,\varphi_1$ be in $C([0,T],E_n)$ and $\psi_2$ be in $E_{n+\alpha}$. Let also $\nu$ be in $\mathcal{M}^+([0,T])$. We say that $u \in L^1([0,T],E_n)$ is a solution to 
\begin{equation}
\left \{
\begin{array}{ll}
-\partial_t u + H(x,Du) - \Delta u = \nu(t) \psi_1+ \varphi_1, & \mbox{ in }[0,T] \times \R^d\\
u(T,x) = \psi_2, & \mbox{ in } \R^d,
\end{array}
\right.
\label{HJBmeasureargument19septembre2022}
\end{equation}
if, for almost all $t\in [0,T]$, for all $x \in \R^d$,
\begin{align}
\notag u(t,x) &= P_{T-t}\psi_2(x) + \int_0^T \mathds{1}_{(t,T]}(s) P_{s-t} \psi_1(s)(x) d \nu(s) + \int_t^T P_{s-t} \varphi_1(s)(x)ds \\
&- \int_t^T P_{s-t} \left[H(.,Du(s,.))\right](x)ds. 
\label{DefHJB19Mars2023}
\end{align}
\label{DefinitionHJB12Septembre2022MeasureArgument}
\end{definition}
We can remark that $u$ is a solution of  \eqref{HJBmeasureargument19septembre2022} if and only if $v:= u-z $ is a solution to 
\begin{equation}
\left \{
\begin{array}{ll}
-\partial_t v + H(x, Dv+Dz) - \Delta v = 0 &\mbox{ in } [0,T] \times \R^d, \\
v(T,x) =0 & \mbox{ in } \R^d.
\end{array}
\right.
\label{equationforv12fevrier2023}
\end{equation}
where 
\begin{equation}
z(t,x) := P_{T-t}\psi_2(x) + \int_0^T  \mathds{1}_{(t,T]}(s) P_{s-t}\psi_1(s)(x) d \nu(s) + \int_t^T P_{s-t} \varphi_1(s)(x)ds. 
\label{definitiondez19Mars2023}
\end{equation}
Proceeding exactly as in the proof of Theorem \ref{TheoremHJB14Septembre2022}, we find that there exists a unique solution $v \in L^{\infty}([0,T],E_n)$ to \eqref{equationforv12fevrier2023} and it satisfies 
$$ \essup_{t\in [0,T]} \|v(t) \|_n \leq C(\int_0^T \|z(t) \|_n dt ). $$
As a consequence we get the following well-posedness result for \eqref{HJBmeasureargument19septembre2022}.

\begin{theorem}
Suppose that $n \geq 3$. Let $\psi_1,\varphi_1$ be in $C([0,T],E_n)$ and $\psi_2$ be in $E_n$. Let also $\nu$ be in $\mathcal{M}^+([0,T])$. Under these conditions, there is a unique solution $ u \in L^{\infty}([0,T], E_n)$ to \eqref{HJBmeasureargument19septembre2022} in the sense of Definition \ref{DefinitionHJB12Septembre2022MeasureArgument}.
Moreover it satisfies 
$$\essup_{t \in [0,T]}  \| u(t) \|_n \leq C(|\nu|, \sup_{t\in [0,T]} \|\psi_1(t) \|_n, \sup_{t\in [0,T]} \|\varphi_1(t) \|_n, \|\psi_2 \|_n ) ,$$
where $|\nu|$ is the total variation norm of $\nu$.
\label{TheoremHJBMeasureArgument}
\end{theorem}
We will need the following stability result.

\begin{proposition} Assume that $(\nu_m)_{m\geq 1} \in L^{\infty}([0,T])$ converges in $\mathcal{M}^+([0,T])$ toward $\nu$. Let $u_m \in \mathcal{C}([0,T],E_n)$ be the solution to the HJB equation \eqref{HJBmeasureargument19septembre2022} with data $(\nu_m,\psi_1,\varphi_1,\psi_2)$ with $\psi_1, \varphi_1 \in \mathcal{C}([0,T],E_n)$ and $\psi_2 \in E_{n +\alpha}$. Then, for all $(t,x) \in [0,T] \times \R^d$ such that $\nu(\{ t \})=0$, it holds:
$$\lim_{m \rightarrow + \infty} u_m(t,x) = u(t,x), $$
$$ \lim_{m \rightarrow + \infty} Du_m(t,x) = Du(t,x), $$
where $u$ is the only element in its equivalence class of $L^{\infty}([0,T],E_n)$ satisfying \eqref{DefHJB19Mars2023} for all $(t,x) \in [0,T] \times \R^d$.
\label{StabilityGeneral19Septembre2022}
\end{proposition}

\begin{proof}
For all $m \geq 1$, we define $z_m$ according to \eqref{definitiondez19Mars2023} with $\nu$ replaced my $\nu_m$ and we let as well $v_m := u_m - z_m$. On the one hand, for all $m$, $v_m$ satisfies 
\begin{equation}
\left \{
\begin{array}{ll}
-\partial_t v_m  - \Delta v_m = -H(x,Du_m) &\mbox{ in } [0,T] \times \R^d, \\
v_m(T,x) =0 & \mbox{ in } \R^d,
\end{array}
\right.
\end{equation}
and therefore, by classical estimates for the heat equation, for all $\alpha \in (0,1/2)$,
$$ \|v_m \|_{\frac{1+\alpha}{2},1+\alpha} \leq C_1 \sup_{t \in [0,T]} \|H(.,Du_m(t,.)) \|_{1+\alpha} \leq C_2$$
for some $C_1>0$ and some $C=C(\sup_{t\in [0,T]} \|u_m(t) \|_{2+\alpha}) >0$. Using Theorem \ref{TheoremHJBMeasureArgument}, we find that the sequence $(v_m)_{m\geq 1}$ is bounded in $\mathcal{C}^{\frac{1+\alpha}{2},1+\alpha}$. Therefore we can find $\tilde{v} \in \mathcal{C}^{\frac{1+\alpha}{2},1+\alpha}$ such that $v_m|_{[0,T] \times B(0,R)}$ converges to $\tilde{v}|_{[0,T] \times B(0,R)}$ in $\mathcal{C}^{\frac{1+\beta}{2},1+\beta}([0,T] \times B(0,R))$ for all $R >0$ and some $\beta \in (0,\alpha)$.  On the other hand, using Portementeau theorem, we have that 
\begin{equation}
\lim_{m \rightarrow +\infty} z_m(t,x) = z(t,x), \hspace{30pt} \lim_{m\rightarrow +\infty} Dz_m(t,x) =Dz(t,x)
\label{Troptardpourunnom20mars2023}
\end{equation}
for all $(t,x) \in [0,T] \times \R^d$ such that $\nu( \{t \}) =0$. Since $\nu (\{t\}) \neq 0$ for at most a countable number of times $t \in [0,T]$, we can use Lebesgue dominated convergence theorem and pass to the limit, as $m\rightarrow +\infty$ in the expression  
$$ v_m(t,x) = - \int_t^T P_{s-t} \left [ H(.,Dv_m(s,.) + Dz_m(s,.)) \right](x) ds. $$
We conclude that, for all $(t,x) \in [0,T] \times \R^d$
$$ \tilde{v}(t,x) = -\int_t^T P_{s-t} \left [ H(.,D\tilde{v}(s,.) + Dz(s,.)) \right](x) ds. $$
If we let $\tilde{u}:= \tilde{v} + z$, we have that $\tilde{u}$ solves the HJB equation \eqref{HJBmeasureargument19septembre2022}  and, by uniqueness, $\tilde{u}  = u$ in $L^{\infty}([0,T],E_n)$. Therefore $v(t,x) = \tilde{v}(t,x)$ for all $(t,x) \in [0,T] \times \R^d$ and we conclude that $v_m|_{[0,T] \times B(0,R)}$ converges to $v|_{[0,T] \times B(0,R)}$ in $\mathcal{C}^{\frac{1+\beta}{2},1+\beta}([0,T] \times B(0,R))$ for all $R >0$, for some $\beta \in (0,\alpha)$. Together with \eqref{Troptardpourunnom20mars2023}, this is enough to conclude the proof of the proposition.
\end{proof}

\subsection{Optimality conditions in the general case}

We first prove a lemma similar to Lemma \ref{Linearizationlemma}.

\begin{lemma}
Let $(\tilde{m}, \tilde{\omega})$ be a relaxed solution, in the sense of \eqref{CandidateRelaxed},  to the constrained Problem \eqref{Problem}. Then there exist $\nu \in \mathcal{M}^+([0,T])$ and $\eta \in \R^+$ satisfying 

\begin{minipage}{0.48 \textwidth}
\begin{equation}
\Psi(\tilde{m}(t)) = 0, \nu-\mbox{ae}
\label{Exclusionnulemma}
\end{equation}
\end{minipage}
\begin{minipage}{0.48 \textwidth}
\begin{equation}
\eta \Psi(\tilde{m}(T))=0,
\label{Exclusionetalemma}
\end{equation}
\end{minipage}
and such that $(\tilde{m}, \tilde{\omega})$ minimizes
\begin{align}
\notag  J^l(\omega , m) &:= \int_0^T \int_{\R^d}  L\Bigl(x, \frac{d \omega}{dt \otimes dm(t)}(t,x) \Bigr) dm(t)(x)dt \\
 \notag &+ \int_0^T \int_{\R^d} \frac{\delta \Psi}{\delta m}(\tilde{m}(t), x)dm(t)(x)d\nu(t) + \int_0^T \int_{\R^d} \frac{\delta \mathcal{F}}{\delta m}( \tilde{m}(t),x) dm(t)(x)dt  \\
 &+ \int_{\R^d} \left[ \eta \frac{\delta \Psi}{\delta m}(\tilde{m}(T),x) +\frac{\delta \mathcal{G}}{\delta m}(\tilde{m}(T),x) \right] dm(T)(x),
 \label{LinearizedFunctionalGeneralSeptembre2022}
\end{align} 
over the pairs $(m,\omega)$ satisfying \eqref{CandidateRelaxed} and where we set, $J^l (m,\omega)= +\infty$ if $\omega$ is not absolutely continuous with respect to $dt \otimes m(t)$.
\label{LinearizationlemmaGeneralConstraint}
\end{lemma}

\begin{proof}
We take $\epsilon, \delta >0$ and $(m^{\epsilon,\delta},\omega^{\epsilon,\delta})$ solutions to the penalized problems \ref{PenalizedProblem}. As $\epsilon,\delta \rightarrow 0$, $ (m^{\epsilon,\delta},\omega^{\epsilon,\delta})$ converges, up to taking a sub-sequence, in $\mathcal{C}([0,T] , \mathcal{P}_{r}(\R^d)) \times \mathcal{M}([0,T] \times \R^d, \times \R^d)$ for $r \in (1,2)$ to a solution to the constrained problem that we can assume, without loss of generality, to be $(\tilde{m}, \tilde{\omega})$. Now $(m^{\epsilon,\delta},\omega^{\epsilon,\delta})$ is also a solution to the linearized problems of Lemma \ref{Linearizationlemma} for some $\lambda^{\epsilon, \delta} , \beta^{\epsilon, \delta} \in L^{\infty}([0,T] ) \times \R^+$ satisfying the exclusion conditions 

\begin{minipage}{0.48 \textwidth}
\begin{equation*}
\lambda^{\epsilon,\delta}(t)  \left \{ 
\begin{array}{ll} 
  =0  & \mbox{if    } \Psi(m^{\epsilon,\delta}(t)) < 0 \\
  \in [0,1] & \mbox{if    } \Psi(m^{\epsilon,\delta}(t)) = 0 \\
 = 1 & \mbox{if    } \Psi(m^{\epsilon,\delta}(t)) > 0, \\
\end{array}
\right.
\end{equation*}
\end{minipage}
\begin{minipage}{0.48 \textwidth}
\begin{equation*}
\beta^{\epsilon,\delta}  \left \{ 
\begin{array}{ll} 
 =0  & \mbox{if    } \Psi(m^{\epsilon,\delta}(T)) < 0 \\
 \in [0 , 1] & \mbox{if    } \Psi(m^{\epsilon,\delta}(T)) = 0 \\
 =1 & \mbox{if    } \Psi(m^{\epsilon,\delta}(T)) > 0. \\
\end{array}
\right.
\end{equation*}
\end{minipage}

Using the controllability lemma \ref{NewControllabilitySeptembre2022} and arguing as in Lemma \ref{LemmaL1Estimate} we can infer that $\frac{\lambda^{\epsilon,\delta}}{\epsilon}$ is bounded in $L^1([0,T])$ independently from $\epsilon,\delta >0$ and $\frac{\beta^{\epsilon,\delta}}{\delta}$ is also bounded in $\R^+$. Let us take $\nu \in \mathcal{M}^+([0,T])$ to be a limit point of $\frac{\lambda^{\epsilon,\delta}}{\epsilon}$ and $\eta$ a limit point of $\frac{\beta^{\epsilon,\delta}}{\delta}$. It is plain to check that $\Psi(\tilde{m}(t))=0$ for $\nu$-almost all $t\in [0,T]$ and $\eta \Psi(\tilde{m}(T)) =0$. Now we can argue as in the proof of Lemma \ref{Linearizationlemma}, passing to the limit in the linearized problems to conclude that $(\tilde{m} , \tilde{\omega})$ is indeed a minimum of \eqref{LinearizedFunctionalGeneralSeptembre2022}.
\end{proof}

We now take $u \in L^{\infty}([0,T],E_n)$ to be the solution, in the sense of Definition \eqref{DefinitionHJB12Septembre2022MeasureArgument} to 
\begin{equation}
\left \{
\begin{array}{ll}
\displaystyle -\partial_tu +H(x,Du)- \Delta u = \nu(t) \displaystyle \frac{\delta \Psi}{\delta m}(\tilde{m}(t),x) + \frac{\delta \mathcal{F}}{\delta m}(\tilde{m}(t),x) &\mbox{ in } [0,T] \times \R^d, \\
\displaystyle u(T,x) = \eta \frac{\delta \Psi}{\delta m}(\tilde{m}(T),x) + \frac{\delta \mathcal{G}}{\delta m}(\tilde{m}(T),x) & \mbox{ in } \R^d.
\end{array}
\right.
\label{HJBGeneral3Octobre}
\end{equation}

We also assume that $u$ is defined for all $(t,x) \in [0,T] \times \R^d$ (and not just $dt$-almost everywhere) by
\begin{align}
\notag u(t,x) &= \eta P_{T-t}\frac{\delta \Psi}{\delta m}(\tilde{m}(T))(x) + P_{T-t}\frac{\delta \mathcal{G}}{\delta m}(\tilde{m}(T))(x) + \int_0^T \mathds{1}_{(t,T]}(s) P_{s-t} \frac{\delta \Psi}{\delta m}(\tilde{m}(s))(x) d \nu(s) \\
& + \int_t^T P_{s-t} \frac{\delta \mathcal{F}}{\delta m}(\tilde{m}(s))(x)ds - \int_t^T P_{s-t} \left[H(.,Du(s,.))\right](x)ds. 
\label{HJBGeneral3OctobrePartout}
\end{align}

Using an approximation argument and Proposition \ref{StabilityGeneral19Septembre2022}, we have the following duality relation:

\begin{proposition}
Let $u \in L^{\infty}([0,T], E_n)$ be a solution to \eqref{HJBGeneral3Octobre} satisfying \eqref{HJBGeneral3OctobrePartout} for all $(t,x) \in [0,T] \times \R^d$.
Let also $(m,\alpha) \in \mathcal{C}([0,T], \mathcal{P}_2(\R^d)) \times L^2_{dt \otimes dm(t)}([0,T] \times \R^d, \R^d)$ be a solution in the sense of distributions to 
\begin{equation*}
\left \{
\begin{array}{ll}
\partial_t m +\mdiv(\alpha m ) - \Delta m = 0, & \mbox{ in }(0,T) \times \R^d, \\
m(0)=m_0.
\end{array}
\right.
\end{equation*}
Then the following duality formula holds for any $t_1 \in [0,T]$ such that $\nu(\{t_1 \})=0$,
\begin{align}
\notag \int_{\R^d} u(t_1,x)dm(t_1)(x) &= \eta \int_{\R^d} \frac{\delta \Psi}{\delta m}(\tilde{m}(T),x)dm(T)(x) + \int_{\R^d} \frac{\delta \mathcal{G}}{\delta m}(\tilde{m}(T),x)dm(T)(x) \\
\notag &- \int_{t_1}^T \int_{\R^d} \left[ H(x,Du(t,x)) + \alpha(t,x).Du(t,x) \right]dm(t)(x)dt \\
&+ \int_{t_1}^T \int_{\R^d} \frac{\delta \Psi}{\delta m}(\tilde{m}(t),x)dm(t)(x)d\nu(t) + \int_{t_1}^T\int_{\R^d}\frac{\delta \mathcal{F}}{\delta m}(\tilde{m}(t),x)dm(t)(x)dt.
\label{DualityRelation19Sept2022} 
\end{align}

\end{proposition}

We can conclude with the proof of Theorem  \ref{MainTheoremWithoutTransCondPsi}.

\begin{proof}[Proof of Theorem \ref{MainTheoremWithoutTransCondPsi}]
We proceed similarly to the proof of Theorem \ref{TheoremPen}. Take $(\tilde{m},\tilde{\omega})$ a relaxed solution to the constrained problem \ref{Problem}. Let also $u \in L^{\infty}([0,T],E_n)$ be the solution to \eqref{HJBGeneral3Octobre} satisfying \eqref{HJBGeneral3OctobrePartout} with $\nu$ and $\eta$ satisfying respectively \eqref{Exclusionnulemma} and \eqref{Exclusionetalemma}.

Recall that the linearized cost $J^l$ is defined in Lemma \ref{LinearizationlemmaGeneralConstraint}. On the one hand, by definition of $L$, it holds that 
\begin{align*}
J^{l}(\tilde{m}, \tilde{\omega}) &= \int_0^T \int_{\R^d} L \Bigl(x,\frac{d \tilde{\omega}}{dt \otimes d\tilde{m}(t)}(t,x) \Bigr) d\tilde{m}(t)(x)dt \\
& \geq -\int_0^T \int_{\R^d} \left[ \frac{d \tilde{\omega}}{dt \otimes d\tilde{m}(t)}(t,x).D u(t,x) + H(x,Du(t,x)) \right] d\tilde{m}(t)(x)dt
\end{align*}
with equality if and only if
\begin{equation}
 \frac{d \tilde{\omega}}{dt \otimes d\tilde{m}(t)} = -D_pH(x,Du), dt\otimes d\tilde{m}(t)-ae. 
\label{JeSaisPasTrop10Sept2022}
\end{equation}
Being $\Psi(m_0) <0$, it holds that $\nu (\{0 \})=0$ because of the exclusion condition \eqref{Exclusionnulemma} and we can use the duality relation \eqref{DualityRelation19Sept2022} with $t_1=0$ and $\displaystyle \alpha = \frac{d \tilde{\omega}}{dt \otimes d \tilde{m}(t)}$ to conclude that 
$$J^l(\tilde{m},\tilde{\omega}) \geq \int_{\R^d} u(0,x)dm_0(x). $$
On the other hand, we can apply relation \eqref{DualityRelation19Sept2022} to the candidate $(m',-D_pH(x,Du(t,x))m')$ where $m'$ is solution to 
\begin{equation*}
\left \{
\begin{array}{ll}
\partial_tm' - \mdiv(D_pH(x,Du(t,x))m') -\Delta m' = 0, & \mbox{ in } (0,T) \times \R^d \\
m'(0)=m_0.
\end{array}
\right.
\end{equation*}
We get $\displaystyle J^l(m', -D_pH(x,Du(t,x))m') = \int_{\R^d} u(0,x)dm_0(x)$ and we can conclude that  the infimum of the linearized problem is indeed $\displaystyle \int_{\R^d} u(0,x)dm_0(x)$, it is achieved at $(\tilde{m}, \tilde{\omega})$ and \eqref{JeSaisPasTrop10Sept2022} holds true. Collecting the equations satisfied by $u$ and $\tilde{m}$, relation \eqref{JeSaisPasTrop10Sept2022} as well as the exclusion conditions of Lemma \ref{LinearizationlemmaGeneralConstraint}, we get the optimality conditions for the constrained problem. Differentiating in space the equation satisfied by $u$ we find that optimal control belong to $BV_{loc}([0,T] \times \R^d, \R^d) \bigcap L^{\infty}([0,T],\mathcal{C}_b^{n-1}(\R^d,\R^d))$.
\end{proof}

\appendix

\section{Appendix}

\label{sec: Appendix Chapter 2}

\subsection{Existence of relaxed solutions}

\label{Existencerelaxedsolutions}

\begin{proof}[Proof of Proposition \ref{EstimationsbaseFPE18Sept}]

Consider a weak solution of
\begin{equation*}
\left \{
\begin{array}{ll}
dX_t = \alpha(t,X_t)dt + \sqrt{2}dB_t, \\
X_{t=0} = X_0 \sim m_0
\end{array}
\right.
\end{equation*}
such that $\mathcal{L}(X_t) = m(t)$, $\forall t \in [0,T]$. The existence of such a solution is guaranteed by the fact that $(\alpha, m)$ solves the Fokker-Planck equation (see \cite{Trevisan2016} and also Proposition 3.1 in \cite{Daudin2020}). Using Jensen inequality, we get for $t,s \in [0,T]$ with $s<t$ 
\begin{align}
\notag \mathbb{E}( |X_t-X_s|^2 ) & \leq 2 \E \left[ \left| \int_s^t \alpha(u,X_u) du\right|^2 \right] + 4 \E \left[ \left| B_t-B_s \right|^2 \right] \\
\notag & \leq 2 (t-s)^2 \E \left[ \int_s^t \left|  \alpha(u,X_u) \right|^2 \frac{du}{t-s} \right] + 4(t-s) \\
\notag & \leq 2(t-s) \int_0^T \int_{\R^d} |\alpha(t,x)|^2dm(t)(x)dt + 4(t-s)
\end{align}
and therefore 
$$d_2(m(s),m(t)) \leq C\sqrt{t-s}$$ 
for some $\displaystyle C=C(\int_0^T\int_{\R^d} |\alpha(t,x)|^2dm(t)(x)dt)>0$ since $d_2(m(s),m(t)) \leq \mathbb{E}( |X_t-X_s|^2 )^{1/2}$. Taking $s=0$ in the above computation also shows that 
\begin{equation*}
\int_{\R^d} |x|^2dm(t) \leq 2 \mathbb{E}( |X_t-X_s|^2 ) + 2 \int_{\R^d} |x|^2dm_0(x) \leq C
\end{equation*}
for another $\displaystyle C=C(\int_{\R^d}|x|^2dm_0(x), \int_0^T\int_{\R^d} |\alpha(t,x)|^2dm(t)(x)dt)>0$.
\end{proof}

\begin{proof}[Proof of Proposition \ref{CompactnessofcurvesFPE18Sept2022}]

We set $\omega^n = \alpha^nm^n$. By Cauchy-Schwarz inequality we find that the total variation $|\omega^n|$ of $\omega^n$ is uniformly bounded. Indeed we have
\begin{align*}
|\omega^n| &= \int_0^T \int_{\R^d} \left| \frac{d \omega^n}{dt \otimes dm^n(t)}(t,x) \right| dm^n(t)(x)dt \\
&\leq \sqrt{T} \left( \int_0^T \int_{\R^d} \left| \frac{d \omega^n}{dt \otimes dm^n(t)}(t,x) \right|^2 dm^n(t)(x)dt \right)^{1/2}. 
\end{align*}
This estimate together with Proposition \ref{EstimationsbaseFPE18Sept} allow us to use Banach-Alaoglu theorem on the one hand and Ascoli theorem on the other hand and deduce that for all $r \in (1,2)$, up to a subsequence, $(m^n, \omega^n)_{n \in \mathbb{N}}$ converges in $\mathcal{C}([0,T], \mathcal{P}^r(\R^d)) \times \mathcal{M}([0,T] \times \R^d,\R^d)$ to some element $(\tilde{m}, \tilde{\omega})$ of $\mathcal{C}([0,T], \mathcal{P}^r(\R^d)) \times \mathcal{M}([0,T] \times \R^d,\R^d)$. It is straightforward that $\tilde{m}(0)=m_0$ and the fact that $(\tilde{m}, \tilde{\omega})$ satisfies the Fokker-Planck equation is a consequence of the weak-$*$ convergence of measures. Using Theorem 2.34 of \cite{Ambrosio2000} (see also Exemple 2.36) in \cite{Ambrosio2000}) we find that $\omega$ is absolutely continuous with respect to $m(t)\otimes dt$ and 
$$ \int_0^T \int_{\R^d}\left| \frac{d\omega}{dt \otimes dm(t)}(t,x) \right|^2  dm(t)(x)dt \leq \liminf_{n \rightarrow +\infty}  \int_0^T \int_{\R^d} |\alpha_n(t,x)|^2dm^n(t)(x)dt. $$
By Proposition \ref{EstimationsbaseFPE18Sept} again, this shows that $m$ belongs to $\mathcal{C}^{1/2}([0,T] , \mathcal{P}_2(\R^d))$.
\end{proof}

Now we give the proof of Lemma \ref{Existenceofweaksolutions2021}.

\begin{proof}[Proof of Lemma \ref{Existenceofweaksolutions2021}]
The result follows from Proposition \ref{EstimationsbaseFPE18Sept} and Proposition \ref{CompactnessofcurvesFPE18Sept2022}. We consider a minimizing sequence $(m^n,\omega^n)$ satisfying \eqref{CandidateRelaxed} and such that, for all $n \in \mathbb{N}$, $J_{\epsilon,\delta}(m^n,\omega^n) \leq  \inf J_{\epsilon,\delta}(m^n,\omega^n) +  1$. By coercivity of $H$ and therefore -by taking convex conjugates- of $L$ we find that there is $C_1>0$ such that, for all $n \in \mathbb{N}$,

\begin{equation}
\int_{\R^d} \int_0^T \left| \frac{d \omega^n}{dt \otimes dm^n(t)} (t,x)\right|^2dm^n(t)(x)dt \leq C_1.
\label{EstimateExistence}
\end{equation}
Using that $(m^n,\omega^n)$ satisfies the Fokker-Planck equation and $m_0$ belongs to $\mathcal{P}_2(\R^d)$ we deduce from Proposition \eqref{CompactnessofcurvesFPE18Sept2022} that, for all $r \in (1,2)$, up to a subsequence, $(m^n, \omega^n)_{n \in \mathbb{N}}$ converges in $\mathcal{C}([0,T], \mathcal{P}^r(\R^d)) \times \mathcal{M}([0,T] \times \R^d,\R^d)$ to some element $(\tilde{m}, \tilde{\omega})$ of $\mathcal{C}([0,T], \mathcal{P}^2(\R^d)) \times \mathcal{M}([0,T] \times \R^d,\R^d)$ which satisfies the Fokker-Planck equation with initial position $m(0)=m_0$. To conclude we use Theorem 2.34 of \cite{Ambrosio2000} to prove that 
$$J_{\epsilon, \delta}(\tilde{m},\tilde{\omega}) \leq \liminf_{n \rightarrow + \infty} J_{\epsilon, \delta}(m_n, \omega_n).$$ 
Therefore $(\tilde{m}, \tilde{\omega})$ is indeed a minimum of $J_{\epsilon, \delta}$.
\end{proof}

\subsection{Technical Results about the HJB equation}

\label{sec: Technical Results about the HJB equation}

We start with a (slightly unusual) version of Grönwall lemma. 

\begin{lemma}
Assume that $l : [0,T] \rightarrow \R^+$ is a bounded measurable map which satisfies, for some  $C_1,C_2>0$ and 
\begin{equation} 
l(t) \leq C_1 + C_2 \int_t^T \frac{l(s)}{\sqrt{s-t}}ds.
\label{GronwallBizarre}
\end{equation}
Then, for almost all $t \in [0,T]$,
$$ l(t) \leq C_1(1+C_2\sqrt{\pi}\sqrt{T-t})e^{C_2^2\pi(T-t)}.$$
\label{SlightlyMoreGeneralGronwallLemma}
\end{lemma} 

\begin{proof}
Arguing by induction, using \eqref{GronwallBizarre} we find that, for all $t \in [0,T]$ and all $n \in \mathbb{N}^*$, it holds
$$ l(t) \leq C_1 \left( 1+ \sum_{k=1}^n C_2^k I_k(t) \right) + \|l \|_{\infty} C_2^{n+1} I_{n+1}(t) $$
where $I_k : [0,T] \rightarrow \R$ is defined for all $k \in \mathbb{N}^*$ by
\begin{equation*}
I_k(t) = \int_t^T \int_{t_1}^T \dots \int_{t_{k-1}}^T \frac{1}{\sqrt{t_1-t} \dots \sqrt{t_k-t_{k-1}} } dt_1 \dots t_k.
\end{equation*} 
Once we have found by induction that, for all $k \geq 1$ and $t \geq 0$, $\displaystyle I_k(t) = \frac{\pi^{k/2}}{\Gamma(k/2+1)} (T-t)^{k/2},$
where $\Gamma$ is Euler's Gamma function, we conclude by elementary computations. 
\end{proof}
%
%
%
%
%
\begin{lemma} 
Assume that $u \in \mathcal{C}([0,T],E_n)$ is a solution to the HJB equation \eqref{HJB12septembre2022} with $f \in \mathcal{C}([0,T],E_n)$ and $g \in E_n$. Then 
$$ \sup_{(t,x) \in [0,T] \times \R^d} |Du(t,x) | \leq C( \int_0^T \|f(t) \|_1dt, \|g \|_1 ). $$
\label{LipschitzEstimateLinfinityBernstein21Sept2022}
\end{lemma}
\begin{proof}
We use the classical Bernstein method. Let $\mu >0$ and $w(t,x):= \frac{1}{2} e^{\mu t} |D u(t,x) |^2 $. Being $f$ in $\mathcal{C}([0,T],E_n)$, $u$ is smooth in space and satisfies the HJB equation in the strong sense. Differentiating the equation with respect to $x$ and taking the scalar product with $e^{\mu t}Du(t,x)$ gives
\begin{align*}
& -\partial_t w(t,x) + Dw(t,x).D_pH(x,Du(t,x)) - \Delta w(t,x)  \\
 &= -\mu w(t,x) - D_xH(x,Du(t,x)).e^{\mu t}Du(t,x) + Df(t,x).e^{\mu t} Du(t,x) - e^{\mu t} |D^2u(t,x)|^2. 
 \end{align*}
Now, by assumption on $H$, $|D_xH(x,Du(t,x)) | \leq C_0(1+ |Du(t,x)|)$ and therefore, for $\mu = 2C_0$,
\begin{align*}
-\partial_t w(t,x) &+ Dw(t,x).D_pH(x,Du(t,x)) - \Delta w(t,x)  \leq C_0e^{\mu t} |Du(t,x)| +  Df(t,x).e^{\mu t} Du(t,x) \\
&\leq \sqrt{2}e^{C_0 T}\left(C_0 +  \|f(t) \|_1 \right) \sup_{(s,y) \in [0,T] \times \R^d} \sqrt{w(s,y)}. 
\end{align*}
By comparison between $w$ and the obvious super-solution 
$$(t,x) \mapsto \frac{1}{2}e^{2C_0 T} \|g \|^2_1 + \sqrt{2}e^{C_0 T} \sup_{(s,y) \in [0,T] \times \R^d} \sqrt{w(s,y)} \int_t^T\left(C_0+ \|f(s) \|_1 \right) ds $$
 we deduce that, for all $(t,x) \in [0,T] \times \R^d$,
 $$ w(t,x) \leq C(1+ \sup_{(s,y) \in [0,T] \times \R^d} \sqrt{w(s,y)} )$$
 for some $\displaystyle C=C(\int_0^T \|f(t) \|_1 dt, \|g\|_1) >0$.
 And therefore, $\sup_{(t,x) \in [0,T] \times \R^d} |Dw(t,x) | \leq C$ for another constant $\displaystyle C= C(\int_0^T \|f(t) \|_1 dt, \|g\|_1) >0$. 
\end{proof}
\color{black}
\begin{lemma}
Assume that $u \in \mathcal{C}([0,T],E_n)$ is a solution to the HJB equation with data $f \in L^1([0,T],E_n)$ and $g \in E_n$ and assume that $u$ satisfies the estimate of the previous lemma then
$$ \sup_{t\in [0,T]} \|u(t) \|_n \leq C( \int_0^T \|f(t)\|_n, \|g \|_n). $$
\label{LemmaHigherOrderEtimates3Octobre2022}
\end{lemma}
\begin{proof}
For all $(t,x) \in [0,T] \times \R^d$, it holds that 
\begin{align*}
|u(t,x) | &\leq |P_{T-t}g(x) | + \int_t^T |P_{s-t}f(s)(x)|ds  + \int_t^T |P_{s-t}\left[H(.,Du(s,.))\right](x)|ds \\
&\leq 2\sqrt{T} \|g \|_0(1+|x|) +2\sqrt{T} \int_t^T \|f(s) \|_0(1+|x|)ds + C(1 + \sup_{(t,x) \in [0,T] \times \R^d} |Du(s,x)| )
\end{align*}
for some $C = C( \sup_{(t,x) \in [0,T] \times \R^d)} |Du(t,x) |) >0$. Above we use the fact that $ \sup_{x \in \R^d} |P_t g(x) | \leq \sup_{x \in \R^d} |g(x)| $ for a bounded function $g$ and $ \sup_{x \in \R^d} \frac{|P_t g(x) |}{1+|x|} \leq 2\sqrt{T} \sup_{x \in \R^d} \frac{|g(x)|}{1+|x|} $ for a function $g$ with linear growth. Since $u$ is assumed to satisfy the Lipschitz estimate  of the previous lemma \ref{LipschitzEstimateLinfinityBernstein21Sept2022}, it holds that
$$ \sup_{t \in [0,T] } \|u (t)\|_1 \leq C( \int_0^T \|f(t) \|_1 dt, \|g(t) \|_1).$$
Now we proceed with higher order derivatives and we argue by induction. Take $k \geq 2$ and assume that we have shown that 
$$ \sup_{t \in [0,T] } \|u (t)\|_{k-1} \leq C( \int_0^T \|f(t) \|_{k-1} dt, \|g(t) \|_{k-1}).$$
Using the inequality $\sup_{ x \in \R^d} |DP_tg (x) | \leq \frac{C}{\sqrt{t}} \sup_{x \in \R^d} |g(x)|$ we get
\begin{align*}
|D^ku(t,x) | &\leq |P_{T-t}D^kg(x) | + \int_t^T |P_{s-t}D^kf(s)(x)|ds  + \int_t^T |DP_{s-t}D^{k-1}\left[H(.,Du(s,.))\right](x)|ds \\
&\leq  \|g \|_k + \int_t^T \|f(s) \|_kds + C\int_t^T \frac{ \sup_{x \in \R^d} |D^{k-1}\left[ H(x,Du(s,x)) \right] |}{\sqrt{s-t}}ds .
\end{align*}
But we can find a constant $C= C(\sup_{t \in [0,T]} \|u(t) \|_{k-1})$ such that 
$$ \sup_{x \in \R^d} |D^{k-1} H(x,Du(s,x)) | \leq C(1 + \sup_{x \in \R^d} |D^ku(s,x) |)$$
and therefore, by Grönwall's lemma \ref{SlightlyMoreGeneralGronwallLemma},
$$ \sup_{(t,x) \in [0,T] \times \R^d} |D^ku(t,x) | \leq C( \|g\|_k, \int_0^T \|f(t) \|_kdt , \sup_{t\in [0,T]} \|u(t) \|_{k-1} )$$ 
and we conclude by induction.
\end{proof}

Following similar computations we can prove the following stability result.

\begin{lemma}
Take $f_1,f_2 \in L_1([0,T],E_n)$ and $g_1,g_2 \in E_n$. Suppose that, $u_1,u_2 \in \mathcal{C}([0,T],E_n)$ are solutions to the HJB equation \eqref{HJB12septembre2022} with data $(f_1,g_1),(f_2,g_2)$ respectively and satisfy the estimate of Lemma \ref{LemmaHigherOrderEtimates3Octobre2022}. Then 
$$ \sup_{t \in [0,T]} \|u_1(t) -u_2(t) \|_n \leq C \Bigl(\int_0^T \|f_1(t)-f_2(t) \|_ndt + \|g_1 - g_2 \|_n \Bigr). $$
for some $\displaystyle C=C\bigl(\int_0^T \|f_1(t) \|_ndt, \int_0^T \|f_2(t) \|_ndt, \|g_1\|_n, \|g_2\|_n\bigr) >0$.
\label{lem:stability15Mai2023}
\end{lemma}

\begin{proof}
For all $(s,x) \in [0,T] \times \R^d$ we can write
\begin{align*}
H(x,Du_1(s,x)) &- H(x,Du_2(s,x)) \\
&= (Du_1(s,x)-Du_2(s,x)) . \int_0^1 D_pH(x,rDu_1(s,x) +(1-r)Du_2(s,x))dr 
\end{align*}
and deduce that, for all $k \geq 1$,
$$ \sup_{x \in \R^d} |D^{k-1}\left[H(x,Du_1(s,x)) - H(x,Du_2(s,x)) \right] | \leq C \|u_1(s) - u_2(s) \|_k 
$$
for some $C = C( \|u_1(s) \|_k, \|u_2(s) \|_k)>0$. The proof of the lemma follows from this observation and the same computations as the proof of Lemma \ref{LemmaHigherOrderEtimates3Octobre2022}.
\end{proof}

\begin{lemma}
Assume that $u \in L^{\infty}([0,T],E_n)$ solves the HJB equation with data $(f,g) \in \mathcal{C}([0,T],E_n) \times E_{n+\alpha}$ then $u$ belongs to $\mathcal{C}([0,T],E_n)$.
\end{lemma}

\begin{proof}
Let us take $k \in \llbracket 1,n \rrbracket $. We fix $h >0$. For $t \in [0,T-h]$ it holds
\begin{align*}
D^ku(t+h,x) &- D^ku(t,x) = P_{T-t-h}D^kg(x) - P_{T-t}D^kg(x) \\
&+ \int_{t+h}^TP_{s-t-h}D^kf(s)(x)ds - \int_t^T P_{s-t}D^kf(s)(x)ds \\
& + \int_{t+h}^T DP_{s-t-h}D^{k-1}H(.,Du(s,.))(x) ds
- \int_t^TDP_{s-t}D^{k-1}H(.,Du(s,.))(x)ds \\
&= \Delta_1 + \Delta_2 + \Delta_3. 
\end{align*}
We estimate the three differences as follows:
$$|\Delta_1| =  | P_{T-t-h}D^kg(x) - P_{T-t}D^kg(x) | \leq |D^kg(x) - P_hD^kg(x) | \leq h^{\alpha/2} ||g ||_{k+\alpha}. $$
Now for the term involving $f$:
\begin{align*}
|\Delta_2 | &= |\int_{t+h}^TP_{s-t-h}D^kf(s)(x)ds - \int_t^T P_{s-t}D^kf(s)(x)ds| \\
&= | \int_t^{T-h} P_{s-t}D^kf(s+h)(x)ds - \int_t^T P_{s-t}D^kf(s)(x)ds | \\
&= |\int_t^{T-h} P_{s-t}(D^kf(s+h)-D^kf(s))(x)ds - \int_{T-h}^T P_{s-t}D^kf(s)(x)ds | \\
&\leq \int_0^{T-h} \|f(s+h) - f(s) \|_k ds + C\sqrt{h} \sup_{t\in [0,T]} \|f(t) \|_{k-1}. 
\end{align*}
Finally for the term involving the Hamiltonian
\begin{align*}
|\Delta_3| &= |\int_{t+h}^T DP_{s-t-h}  D^{k-1}H(.,Du(s,.))(x) ds
- \int_t^TDP_{s-t}D^{k-1}H(.,Du(s,.))(x)ds| \\
&=| \int_t^{T-h} DP_{s-t}D^{k-1}\left[ H(.,Du(s+h,.) - H(.,Du(s,.)) \right](x)ds \\
&- \int_{T-h}^T DP_{s-t}D^{k-1}\left[H(.,Du(s,.))\right](x)ds | \\
&\leq C(\essup_{t\in [0,T]} \|u(t) \|_k) \int_t^{T-h} \frac{\sup_{x \in \R^d} |D^ku(s+h,x)-D^ku(s,x) |}{\sqrt{s-t}}ds \\
&+ C(\essup_{t\in [0,T]} \|u(t) \|_k) \sqrt{h} \\
& \leq C(\essup_{t\in [0,T]} \|u(t) \|_k )(\sqrt{h} + \int_t^{T-h} \frac{\|u(s+h) - u(s) \|_k}{\sqrt{s-t}} ds ).
\end{align*} 
Using again Grönwall Lemma \ref{SlightlyMoreGeneralGronwallLemma}, we get, for all $t \in [0,T]$,
\begin{align*}
 \|u(t+h) - u(t) \|_n & \leq C( \essup_{t\in [0,T]} \|u(t) \|_n ) ( h^{\alpha/2} \|g \|_{n+\alpha} + \int_0^{T-h} \|f(s+h) - f(s) \|_n ds \\
 &+ \sqrt{h} \sup_{t\in [0,T]} \|f(t) \|_{n-1} ).
 \end{align*}
 Being $f$ in $\mathcal{C}([0,T],E_n)$, the right-hand side converges to $0$ when $h$ goes to $0$ and therefore
 $$\lim_{h \rightarrow 0} \sup_{t\in [0,T-h]} \|u(t+h) - u(t) \|_n =0$$
 which concludes that $u$ belongs to $\mathcal{C}([0,T],E_n)$.
\end{proof}

As a consequence, we get the existence of solutions from the classical case.

\begin{proposition}
Take $f \in L^1([0,T] , E_n)$ and $g \in E_{n+\alpha}$. Then there exists a unique solution in $u \in \mathcal{C}([0,T],E_n)$ to the HJB equation with data $(f,g)$ and it satisfies the estimate of  Lemma \ref{LemmaHigherOrderEtimates3Octobre2022}.
\label{Proposition15Mai2023}
\end{proposition}

\begin{proof}[Proof of Proposition \ref{Proposition15Mai2023}]
We take a sequence of smooth functions $f_m:[0,T] \times \R^d \rightarrow \R$ and $g_m:\R^d \rightarrow \R$ converging respectively to $f$ in $L^1([0,T],E_n)$ and to $g$ in $E_{n+\alpha}$. For each $m$, the existence of a strong solution $u_m \in \mathcal{C}([0,T],E_n)$ follows from Schauder theory and our a priori Lipschitz estimate. Thanks to the previous lemma, we know that $u_m$ is a Cauchy sequence in $L^{\infty}([0,T],E_n)$ and therefore it converges in this space to some $u$. The subspace $\mathcal{C}([0,T],E_n)$ being closed in $L^{\infty}([0,T],E_n)$ we have that $u$ belongs to $\mathcal{C}([0,T],E_n)$. We can also pass to the limit in the equation
$$ u_m(t,x) = P_{T-t}g_m(x) + \int_t^T P_{s-t} f_m(s)(x)ds - \int_t^T P_{t-s} \left[ H(.,Du_m(s,.)) \right](x)ds $$
to conclude that $u$ is a solution.

The uniqueness of solutions is a straightforward consequence of the stability estimate of  Lemma \ref{lem:stability15Mai2023}.
\end{proof}

We are finally ready to prove Theorem \ref{TheoremHJB14Septembre2022}.

\begin{proof}[Proof of Theorem \ref{TheoremHJB14Septembre2022}]
Combining Proposition \ref{Proposition15Mai2023} and Lemma \ref{lem:stability15Mai2023} we get Theorem \ref{TheoremHJB14Septembre2022}.
\end{proof}

\paragraph{Acknowledgment} The author wishes to thank Pierre Cardaliaguet for suggesting the problem and for fruitful discussions during the preparation of this work. 

\bibliographystyle{plain}
\bibliography{/Users/sam/Documents/Bibtex/JMPA2022Daudin.bib}

\end{document}